\documentclass[3p,sort&compress]{personalartcl}

\usepackage[utf8]{inputenc}

\usepackage{amsmath,amscd,amsfonts,amssymb,amsthm}

\usepackage{graphicx}
\usepackage[usenames]{color}

\usepackage{tikz}

\usepackage{verbatim,shortvrb}

\usepackage[colorlinks]{hyperref}
\hypersetup{%
colorlinks,
pdfborder= 1 0 1,
linkcolor=black, 
anchorcolor=black, 
citecolor=black, 
filecolor=black, 
menucolor=black, 
urlcolor=black
}

\usepackage{enumerate}

\usepackage{lineno}

\usepackage{url}
\makeatletter
\def\url@leostyle{%
  \@ifundefined{selectfont}{\def\UrlFont{\sf}}{\def\UrlFont{\small\ttfamily}}}
\makeatother
\usepackage{mathrsfs} 

\newlength{\defbaselineskip}
\newcommand{\setlinespacing}[1]%
           {\setlength{\baselineskip}{#1 \defbaselineskip}}

\journal{ArXiv}

\usepackage{tikz-cd}
\usetikzlibrary{matrix,arrows}
\usepackage{bbm}
\allowdisplaybreaks

\newtheorem{theorem}{Theorem}[section]
\newtheorem{lemma}[theorem]{Lemma}
\newtheorem{corollary}[theorem]{Corollary}
\newtheorem{proposition}[theorem]{Proposition}
\newtheorem{claim}[theorem]{Claim}
\theoremstyle{definition}
\newtheorem{definition}[theorem]{Definition}
\newtheorem{remark}[theorem]{Remark}
\newtheorem{example}[theorem]{Example}
\newtheorem{notation}[theorem]{Notation}

\newtheorem*{acks}{Acknowledgements}

\newcommand{\N}{\ensuremath{\mathbb{N}}}
\newcommand{\Z}{\ensuremath{\mathbb{Z}}}
\newcommand{\R}{\ensuremath{\mathbb{R}}}
\newcommand{\CC}{\ensuremath{\mathbb{C}}}
\renewcommand{\int}{\ensuremath{[0,1]}}

\renewcommand{\leq}{\leqslant}
\renewcommand{\geq}{\geqslant}

\newcommand{\m}{\mathfrak{m}}
\newcommand{\h}{\mathfrak{h}}
\DeclareMathOperator{\rad}{Rad}

\DeclareMathOperator{\Max}{Max}
\newcommand{\id}[1]{{\rm Id}_{#1}}

\DeclareMathOperator{\fm}{\mathnormal{f}_{\frac{1}{2}}}
\DeclareMathOperator{\fmn}{\mathnormal{f}_{\frac{1}{2^\mathnormal{n}}}}
\DeclareMathOperator{\fmnn}{\mathnormal{f}_{\frac{1}{2^\mathnormal{n-1}}}}
\DeclareMathOperator{\fmi}{\mathnormal{f}_{\frac{1}{2^\mathnormal{i}}}}
\DeclareMathOperator{\fmm}{\mathnormal{f}_{\frac{1}{2^\mathnormal{m}}}}

\DeclareMathOperator{\fix}{Fix} 
\newcommand{\forg}[1]{\ensuremath{\lvert#1\rvert}} 

\renewcommand{\d}{{\rm d}} 

\renewcommand{\H}{{\rm H}} 

\DeclareMathOperator{\C}{{\rm C}}

\newcommand{\KH}{\ensuremath{\sf K}} 
\newcommand{\Top}{\ensuremath{\sf Top}}  
\newcommand{\St}{\ensuremath{\sf St}} 

\newcommand{\MV}{\ensuremath{\sf MV}}
\newcommand{\Set}{\ensuremath{\sf Set}}
\newcommand{\Cst}{\ensuremath{\sf C^*}}
\newcommand{\Grp}{\ell \sf G_u} 
\newcommand{\Va}{\ensuremath{\sf \Delta}} 

\newcommand{\2}{{\mathbbm 2}}
\newcommand{\I}{{I}} 
\newcommand{\Tb}{{\sf T}_{\2}}
\newcommand{\Td}{{\sf T}}

\DeclareMathOperator{\V}{\mathbb{V}} 

\begin{document}
\begin{frontmatter}
\title{Stone duality above dimension zero:\\ Axiomatising the algebraic theory of $\C(X)$}
\author[milano]{Vincenzo Marra}
\ead{vincenzo.marra@unimi.it}
\author[milano,parigi]{Luca Reggio}
\ead{luca.reggio@irif.fr}
\address[milano]{Dipartimento di Matematica ``Federigo Enriques'', Universit\`a degli Studi di Milano, via Cesare Saldini 50, 20133 Milano, Italy}
\address[parigi]{Universit\'e Paris Diderot, Sorbonne Paris Cit\'e, IRIF, Case 7014, 75205 Paris Cedex 13, France}


\begin{keyword} Algebraic theories \sep Stone duality \sep Boolean algebras \sep MV-algebras  \sep Lattice-ordered Abelian groups \sep $\mathrm{C}^*$-algebras \sep  Rings of continuous functions \sep  Compact Hausdorff spaces \sep Stone-Weierstrass Theorem \sep Axiomatisability.
\MSC[2010]{Primary: 03C05. 
Secondary: 06D35 \sep 06F20 \sep 46E25.}
\end{keyword}
\begin{abstract}
It has been known since the work of Duskin and Pelletier four decades ago that $\KH^{\rm op}$, the opposite of the category of compact Hausdorff spaces and continuous maps, is 
monadic over the category of sets. It follows that $\KH^{\rm op}$ is equivalent to a possibly infinitary variety of algebras $\Va$  in the sense of S{\l}omi\'nski and Linton. 
Isbell  showed in 1982 that the Lawvere-Linton algebraic theory of $\Va$ can be generated using a finite number of finitary operations, together with a single operation of 
countably infinite arity. In 1983, Banaschewski  and Rosick\'y independently proved a conjecture of Bankston,  establishing a strong negative result on the axiomatisability of $\KH^{\rm op}$. 
 In particular,  $\Va$ is not a finitary variety --- Isbell's result is best possible.  The problem of axiomatising $\Va$ by 
equations has remained open. Using the theory of Chang's MV-algebras as a key tool, along with Isbell's fundamental insight on the semantic 
nature of the infinitary operation, we provide a \emph{finite} axiomatisation of $\Va$. 
\end{abstract}
\end{frontmatter}
\section{Introduction.}
In the category  $\Set$ of sets and functions, coordinatise the two-element set as $\2:=\{0,1\}$. Consider the full subcategory $\Tb$ of $\Set$ whose objects are the finite Cartesian powers $\2^{n}$. 
Equipped with distinguished projection functions $\2^{n}\to \2$ to play the r\^{o}le of variables, $\Tb$ is  
the Lawvere algebraic theory \cite{lawvere63} of Boolean algebras. That is, the category $[\Tb,\Set]_{\rm fpp}$ of all functors from $\Tb$ to $\Set$ that 
preserve finite products is the variety of Boolean 
algebras. This variety is finitary, as witnessed by the fact that $n$  only ranges over  finite sets. The full subcategory of free finitely 
generated Boolean algebras is then dually equivalent to $\Tb$, by (Lawvere's) construction; consequently, the full subcategory of finitely 
presented Boolean 
algebras is dually equivalent to the category $\Set_{f}$ of finite sets. Quite generally, any  
finitary variety is equivalent to  Grothendieck's Ind-completion of its full subcategory of finitely presented algebras \cite[Corollary VI.2.2]{j}.  Therefore the category of 
Boolean algebras is dually equivalent to the Pro-completion of $\Set_{f}$, which,  by a folklore theorem, can be identified with  the category $\St$ 
of compact Hausdorff zero-dimensional  spaces and the continuous maps between them. This is Stone's 
celebrated duality 
theorem for Boolean algebras \cite{st2}. The objects of $\St$ are called \emph{Stone spaces}.

\smallskip Now enlarge $\2$ to the real unit interval $\I:=\int\subseteq \R$, and regard $\I$ endowed with its Euclidean topology as an object in the category $\Top$ of topological spaces and continuous maps. Consider the full subcategory $\Td$ of $\Top$ whose 
objects are the countable Cartesian powers $\I^{\kappa}$ equipped with the product (Tychonoff) topology. With distinguished projections $\I^{\kappa}\to \I$,
the category $\Td$  is then the Lawvere-Linton algebraic theory \cite{Linton66} of some variety of algebras. That is, the category $[\Td,\Set]_{\rm cpp}$ of all functors from $\Td$ to $\Set$ that preserve countable products is a variety; but this time the 
variety, as described here, is \emph{infinitary} \cite{Slominsky:1959}, in that $\kappa$ need not be finite.  Thus, just like Boolean algebras are the models in $\Set$ of the  algebraic theory of the finite Boolean cubes $\2^{n}$, so are the algebras in $[\Td,\Set]_{\rm cpp}$ the models in $\Set$ of the infinitary algebraic theory of the metrisable Tychonoff cubes $I^{\kappa}$.

\smallskip  In a terse and  important paper \cite{Isbell} (see also \cite{Semadeni:1986}), Isbell showed that $[\Td,\Set]_{\rm cpp}$ is dually equivalent to 
$\KH$, the full subcategory of $\Top$ whose objects are compact and Hausdorff spaces. Previously, Duskin \cite[5.15.3]{Duskin1969} had shown that, 
indeed, $\KH^{\rm op}$ must be equivalent to some possibly infinitary variety of algebras. Isbell's result further shows that such a variety can be defined 
using finitely many finitary operations, along with a \emph{single} operation of countably infinite arity. 
The algebras in $[\Td,\Set]_{\rm cpp}$  freely generated by a set $S$ have $\C(I^{S},I)$, the collection of all continuous functions $I^S \to I$, as their 
 underlying sets. Isbell's operation is then the uniformly convergent series
\begin{align}\label{e:isbell}
\sum_{i=1}^{\infty}\frac{f_{i}}{2^{i}}, \quad \text{for } f_{i}\in \C(I^S,I).
\end{align}

Since $[\Td,\Set]_{\rm cpp}$ is a variety, it can be axiomatised by equations. The problem of finding a manageable axiomatisation of $[\Td,\Set]_{\rm cpp}$, if one such exists, has remained open. Our main result in this paper is a \emph{finite} equational axiomatisation of $[\Td,\Set]_{\rm cpp}$.

\begin{remark}\label{rem:subd-repr}
Any variety ${\sf V}$ dually equivalent to $\KH$ enjoys rather special properties. First, the (initial) algebra $\mathrm{F}(0)$ freely generated by the 
empty set generates the variety --- that is, each algebra in ${\sf V}$ is a homomorphic image of a subalgebra of a power of $\mathrm{F}(0)$. 
(We note in passing an interesting consequence: ${\sf V}$ has no proper non-trivial subvarieties). 
In fact, in category-theoretic terminology, $\mathrm{F}(0)$ is a cogenerator in ${\sf V}$. For readers familiar with the terminology of universal algebra, 
this entails that $\mathrm{F}(0)$ generates the variety ${\sf V}$ as a quasivariety.
An even stronger result holds.
Birkhoff proved in \cite{birkhoff1944} that each  algebra of a \emph{finitary} variety is a subdirect product of subdirectly irreducible algebras\footnote{An algebra $A$ is a \emph{subdirect product} of the family $\{A_i\}_{i\in I}$ of algebras if there is a monomorphism $\alpha\colon A\to\prod_{i\in I}A_i$ such 
that $\pi_i\circ\alpha(A)=A_i$ for all $i\in I$, where $\pi_i\colon\prod_{i\in I}A_i\to A_i$ is the $i^{\rm th}$ projection. Then $A$ is \emph{subdirectly irreducible} if, whenever it is represented as the subdirect product of a family $\{A_i\}_{i\in I}$, there is $i\in I$ such that $A$ is isomorphic to $A_i$. \label{fn:sub-direct}} in the variety. Importantly, he pointed out that this result does not  extend to infinitary varieties \cite[p.\ 765]{birkhoff1944}. 
It can be shown by means of topological arguments in $\KH$ that every algebra in ${\sf V}$ is a subdirect product of copies of $\mathrm{F}(0)$, and the latter is simple and therefore subdirectly irreducible. Hence, \emph{Birkhoff's subdirect representation theorem holds for ${\sf V}$}.
\end{remark}

It is not hard to see that $\KH^{\rm op}$ is not equivalent to a finitary variety. Let us prove a  stronger fact (Remark \ref{r:presheaf}), and compare it to results in the literature (Remark \ref{r:ban-ros}).
\begin{lemma}\label{l:knownfacts}The following hold for any compact Hausdorff space $X$.
\begin{enumerate}
\item $X$ is the continuous image of some Stone space.
\item $X$ is a Stone space if, and only if, $X$ is a cofiltered limit in $\KH$ of finite discrete spaces.
\end{enumerate}
\end{lemma}
\begin{proof}
Item 1 is a consequence, for example,  of {\it Gleason's Theorem}: $X$ has a projective topological cover in $\KH$, which is in fact extremally disconnected (=the closure of each open set is again open). See \cite[3.2]{Gleason58}.
Item 2 is folklore,  see e.g.\ \cite[p.\ 236]{j}.
\end{proof}
Recall that an object $A$ of a category\footnote{All categories are assumed to be locally small.} ${\sf C}$ is 
(\emph{Gabriel-Ulmer}) \emph{finitely presentable} if the covariant $\hom$-functor $\hom_{\sf C}{(A,-)}\colon{\sf C}\to\Set$ preserves filtered 
colimits \cite[Definition 6.1]{GabrielUlmer71}, \cite[Definition 1.1]{adameketal94}. Further recall that  ${\sf C}$  is \emph{finitely accessible} provided that it has filtered 
colimits, and that there exists a set $S$ of its objects such that (i) each object of $S$ is finitely presentable, and (ii) each object of {\sf C} 
is a filtered colimit of objects in $S$. See \cite{MakkaiPare89}, \cite[Definition 2.1]{adameketal94}.

\begin{proposition}\label{fin-acc}
Assume that ${\sf F}$ is a full subcategory of $\KH$ extending $\St$. If ${\sf F}$ is dually equivalent to a finitely accessible category --- in 
particular, if ${\sf F}^{\rm op}$ is  a finitary variety --- then ${\sf F}=\St$.
\end{proposition}

\begin{proof} It suffices to show that each object of ${\sf F}$ is a Stone space, for then $\St={\sf F}$.
Since  ${\sf F}^{\rm op}$ is finitely accessible, every object in ${\sf F}$ is the cofiltered limit of finitely copresentable objects. We \emph{claim} that every 
finitely copresentable space $F$ in ${\sf F}$ is finite. By Lemma \ref{l:knownfacts}.(1) there exists a continuous 
surjection $\gamma \colon  G \to F$, with $G$ a Stone space. By Lemma \ref{l:knownfacts}.(2), $G$  is the cofiltered limit in $\KH$ of finite discrete spaces 
$\{G_i\}_{i\in I}$; write $\alpha_i\colon G \to G_i$ for the limit arrows. Note that, since $G$ lies in ${\sf F}$ and the full embedding ${\sf F}\to\KH$ reflects limits, $G$  is in fact the cofiltered limit  of  $\{G_i\}_{i\in I}$ in ${\sf F}$.
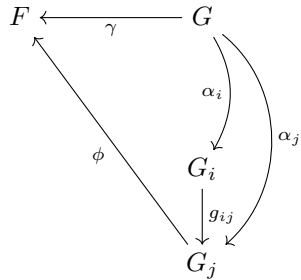
\[ \begin{tikzcd}
F  & & G \arrow{ll}{\gamma} \arrow[bend left=30,yshift=0ex,xshift=0ex]{dd}[swap]{\alpha_i} \arrow[bend left=50,yshift=0ex,xshift=0ex]{ddd}{\alpha_j} \\
& &  \\
& & G_i \arrow{d}{g_{ij}} \\
& & G_j \arrow{uuull}{\phi}
\end{tikzcd}\]
Since $F$ is finitely copresentable in {\sf F}, there exists $j \in I$ together with an {\sf F}-arrow $\phi\colon G_j\to F$  such that $\gamma=\phi\circ\alpha_j$. But $\gamma$ is surjective, and therefore so is $\phi$. Hence $F$ is finite, and our \emph{claim} is settled. Since cofiltered 
limits of finite discrete spaces in ${\sf F}$ are computed as in $\KH$,  each object of ${\sf F}$ is a Stone space by Lemma \ref{l:knownfacts}.(2), as was to be shown. Finally, it is a standard fact that finitary varieties of algebras are finitely accessible, see e.g.\ \cite[Corollary 3.7]{adameketal94}.
\end{proof}
\begin{remark}\label{r:presheaf}In this remark  we assume  familiarity with the terminology of topos theory and categorical logic.
It is well known that finitely accessible categories are precisely the categories of models in $\Set$ of geometric theories of presheaf type. (The 
morphisms are  the homomorphisms, i.e., they  preserve operations and relations.) See for example \cite[0.1]{Beke04}.
Thus, Proposition \ref{fin-acc} entails the corollary: \emph{The category $\KH^{op}$ is not axiomatisable by a geometric theory of presheaf type.}
\end{remark}
\begin{remark}\label{r:ban-ros}A class of finitary algebras is \emph{elementary} if it is the category of models in $\Set$ of a first-order theory, and it is a \emph{{\rm P}-class} if it is closed under arbitrary Cartesian products. Bankston conjectured that 
no elementary {\rm P}-class of algebras is equivalent to $\KH^{\rm op}$, and proved a partial result towards settling the conjecture \cite[Theorem 0.2.(iv)]{Bankston82}. In 1983, Banaschewski and Rosick\'y independently settled Bankston's conjecture,\footnote{We learned 
from Ji\v{r}\'i Rosick\'y (private e-mail communication) that Rosick\'y presented his solution at the Category Theory conference
in Oberwolfach in September 1983, and that after his talk, Banaschewski 
informed the audience of his own solution. \label{fn:bank}} see \cite{Banaschewski84}, {\cite[Theorem 3]{Rosicky89}}. In particular, Banaschewski proved the following result. 
\emph{Assume that ${\sf F}$ is a full subcategory of $\KH$ extending $\St$. If ${\sf F}$ is dually equivalent to an elementary {\rm P}-class of algebras --- 
in particular, if ${\sf F}^{\rm op}$ is  a finitary variety --- then ${\sf F}=\St$.} 
Our Proposition \ref{fin-acc}  and  Banaschewski's result are formally incomparable: neither entails the other. It seems clear, however, that the solution 
to Bankston's conjecture is deeper than Proposition \ref{fin-acc}. To the best of our knowledge it is an open problem whether $\KH^{\rm op}$ is 
axiomatisable by a first-order theory.\end{remark}

\smallskip Our key tool to provide an axiomatisation of $[\Td,\Set]_{\rm cpp}$ is the theory of \emph{MV-algebras}. 
These structures are to \emph{{\L}ukasiewicz's $\int$-valued propositional logic}  as 
Boolean algebras are to classical two-valued logic, cf.\ \cite[Chapter 4]{cdm2000}.
 Chang  \cite{Chang58} first carried out the algebraisation of {\L}ukasiewicz logic, and called the resulting structures MV-algebras 
``for want of a better name'' \cite[p.\ 467]{Chang58}. In \cite{chang2}, using MV-algebras, Chang obtained an 
algebraic proof of the completeness theorem for {\L}ukasiewicz logic. The standard reference for the basic theory of MV-algebras is 
\cite{cdm2000}, whereas \cite{Mundici11} deals with advanced topics. 

This paper is organised as follows. In Section \ref{s:pre} we provide the needed background.
We explain the connection between MV-algebras and lattice-ordered groups through Mundici's functor $\Gamma$ \cite[Theorem 3.9]{Mundici86}, which we later exploit in Section \ref{s:SW} to formulate an MV-algebraic version of the Stone-Weierstrass Theorem.
The problem of axiomatising $\KH^{\rm op}$ first arose in the context of Gelfand-Naimark  duality between $\KH$, and the category of 
complex commutative unital $\mathrm{C}^*$-algebras  \cite[Lemma 1]{GelfandNeumark43}.
We explain the relationship between $\mathrm{C}^*$-algebras and MV-algebras at the end of Section \ref{s:pre}.

In \cite[Section 4]{Cignolietal2004} the authors construct an adjunction between the category $\MV$ of MV-algebras
and $\KH^{\rm op}$, and show that the full subcategory of $\MV$ whose objects are, up to isomorphism, the MV-algebras of all $\int$-valued continuous functions on some compact Hausdorff space is dually equivalent to $\KH$. We summarise their result in 
Section \ref{s:reflection}, indicating how the adjunction fits into (and may be recovered from) the general theory of concrete dual adjunctions as presented in e.g.\ \cite{PT91}. We take the time to stress how the classical approach through maximal ideals and the general approach through dualising objects are related to each other via \emph{H\"{o}lder's Theorem} (Lemma \ref{MVHolder}), a landmark result in the theory of ordered groups. 

In Section \ref{s:delta} we introduce \emph{$\delta$-algebras} 
by expanding the language of MV-algebras with the infinitary operation $\delta$, and by adding finitely many equational axioms to those of MV-algebras. Thus, by definition each $\delta$-algebra has an 
underlying MV-algebra. The operation 
$\delta$, as axiomatised in Definition \ref{d:delta}, is intended to model Isbell's series \eqref{e:isbell}.  Throughout, we write $\Va$ for the category of 
$\delta$-algebras and their homomorphisms. The bulk of Section \ref{s:delta} consists of elementary facts about $\delta$-algebras that are needed for 
the proof of our main result.

In Section \ref{s:semisimple} we prove that any $\delta$-algebra must be semisimple as an MV-algebra (Theorem \ref{semisem}), i.e.\ its radical ideal of ``infinitesimal elements" must be trivial. Using this, in Section \ref{s:etadelta} we show that the forgetful functor $\Va\to\MV$ is full.
In Section \ref{s:SW} we provide the MV-algebraic version of the Stone-Weierstrass Theorem (Lemma \ref{l:mvsw}) by translating the classical result to MV-algebras via Mundici's functor $\Gamma$. In turn, Lemma \ref{l:mvsw} leads us to a Stone-Weierstrass Theorem for $\delta$-algebras (Theorem \ref{t:SW-Delta}). Finally, Section \ref{s:main} assembles these facts together to obtain our main result, Theorem \ref{thm:main}, that $\Va$ is dually equivalent to $\KH$. 

\section{MV-algebras: preliminaries and background.}\label{s:pre}
\begin{notation}\label{n:basic}For the entire paper we set $\N:=\{1,2,\ldots\}$. As usual, $\R$ denotes the set of real numbers,  and $\int\subseteq \R$ the real unit interval.
\end{notation}

An \emph{MV-algebra} is an algebraic structure $(A,\oplus,\neg,0)$, where $0\in A$
is a constant, $\neg$ is a unary operation satisfying $\neg\neg x=x$, $\oplus$ is a binary operation making $(A,\oplus,0)$
a commutative monoid, the element $1$ defined as $\neg 0$ satisfies $x\oplus1=1$, and the law
\begin{align}\label{mvlaw}
 \neg(\neg x \oplus y)\oplus y = \neg(\neg y \oplus x)\oplus x
\end{align}
holds. The operation $\oplus$ is sometimes called \emph{truncated addition}. Any MV-algebra has an underlying structure of
distributive lattice bounded below by $0$ and above by $1$. Joins are defined as $x \vee y := \neg(\neg x \oplus  y)\oplus y$.
Thus, the characteristic law \eqref{mvlaw} states
that $x\vee y=y\vee x$.
Meets are defined by the De Morgan condition  $x \wedge y := \neg (\neg x \vee \neg y)$.
The De Morgan dual of $\oplus$, on the other hand, is the operation traditionally denoted
\[
x \odot y := \neg (\neg x \oplus \neg y)
\]
which, like $\oplus$, is not idempotent. The further derived connective
\[
x\ominus y := x \odot \neg y
\] 
is known as \emph{truncated subtraction}.
Boolean algebras are precisely
those MV-algebras that are idempotent, meaning that $x\oplus x = x$ holds; equivalently, that  $x\odot x = x$ holds; equivalently, that  the {\it tertium non datur} law $x\vee\neg x=1$ holds. Throughout, we write $\MV$ to denote the category whose objects are MV-algebras, and whose arrows are the MV-algebra homomorphisms, i.e.\ the functions preserving $\oplus$, $\neg$, and $0$.

 The real unit interval $\int$ can be made into an MV-algebra with neutral element $0$ by defining 
\begin{align}\label{eq:plus}
x\oplus y := \min{\{x+y,1\}},
\end{align}
and 
\begin{align}\label{eq:not}
\neg x:=1-x.
\end{align} 
The underlying lattice order of this MV-algebra coincides with
the order that $\int$ inherits from the real numbers. When we 
refer to  $\int$ as an MV-algebra, we always mean the  structure just described. \emph{Chang's completeness theorem}  states that $\int$ generates the variety of MV-algebras \cite[2.5.3]{cdm2000}. Explicitly, this means that any MV-algebra is a homomorphic image of a subalgebra of a power of 
$\int$.

\begin{example}\label{ex:c(x)mv}The standard example of MV-algebra, the interval
$\int$  equipped with the constant $0$ and the operations (\ref{eq:plus}--\ref{eq:not}), generalises as follows. If $X$ is any set, the collection $\int^{X}$ of all functions from $X$ to $\int$ inherits the structure of an MV-algebra upon defining operations pointwise. If, additionally, $X$ is a topological space, since $\oplus\colon \int^{2}\to\int$, $\neg\colon\int\to\int$, and $0$ are continuous with respect to the Euclidean topology of $\int$, the subset 
\begin{align}\label{eq:c(x)}
\C(X):=\{f\colon X\to\int\mid f \,\text{ is continuous}\}
\end{align}
is a subalgebra of the MV-algebra $\int^{X}$. The MV-algebra \eqref{eq:c(x)} plays a crucial r\^{o}le in this paper.\footnote{To avoid misunderstandings, let us point out that $\C(X)$ is often used in the $\mathrm{C}^*$-algebraic literature to denote the $\mathrm{C}^*$-algebra of complex-valued continuous functions on the space $X$, and in the literature on rings of continuous functions to denote the ring of real-valued continuous functions on the space $X$. By contrast, we use $\C(X)$ as in \eqref{eq:c(x)} for $\int$-valued continuous functions, $\C(X,\R)$ for real-valued continuous functions, and $\C(X,\CC)$ for complex-valued continuous functions.}
\end{example}

\subsection{Elementary lemmas.}\label{ss:elem}
We collect here the elementary facts about MV-algebras that we need in the sequel. All proofs may be found in \cite{cdm2000} or \cite{Mundici11}, as specified below.
\begin{lemma}[Underlying lattice {\cite[1.1.2, 1.1.5]{cdm2000}}] \label{mv1} 
Let $A$ be an MV-algebra, and  let $x,y\in A$. The following are equivalent.
\begin{enumerate}
\item $\neg x\oplus y=1$.
\item $x\ominus  y=0$.
\item $y=x\oplus (y\ominus x)$.
\item There exists $z\in A$ such that $y=x\oplus z$.
\end{enumerate}
Upon defining  $x\leq y$ if $x,y\in A$ satisfy one of the equivalent conditions above, $(A,\leq)$ is a  lattice whose join and meet operations are given by $x \vee y := \neg(\neg x \oplus  y)\oplus y$ and $x \wedge y := \neg (\neg x \vee \neg y)$, respectively. Moreover, the lattice $(A,\leq)$ is bounded below by $0$ and above by $1$.\qed
\end{lemma}
\begin{lemma}[Interplay of algebra and order {\cite[1.1.4]{cdm2000}}]\label{mv2} 
Let $A$ be an MV-algebra, and  let $x,y, w,z\in A$. If $x\leq y$ and $w\leq z$, then $x\odot w\leq y\odot z$ and  $x\oplus w\leq y\oplus z$.
\qed \end{lemma}
\begin{remark} The reference \cite[1.1.4]{cdm2000} only proves Lemma \ref{mv2} for $w=z$. The generalisation stated above is straightforward.
\end{remark}
Chang's \emph{distance function} on the MV-algebra $A$ is the function $\d\colon A\times A\to A$ defined by \[\d(x,y):=(x\ominus y)\oplus (y\ominus x)\] for every $x,y\in A$. In the case of the  MV-algebra of continuous functions \eqref{eq:c(x)}, $\d(f,g)$ is the function $|f-g|$. In particular, in the MV-algebra $\int$, the 
Chang's distance  is the usual Euclidean distance.

\begin{lemma}[Chang's distance {\cite[1.2.5]{cdm2000}}] \label{mv3} 
If $A$ is an MV-algebra, and $x,y\in A$ satisfy $x\leq y$, then $y=x\oplus \d(x,y)$.
\qed \end{lemma}
The monoidal operation $\oplus$ is not cancellative. However:
\begin{lemma}[MV-cancellation law {\cite[p.\ 106]{Mundici11}}]\label{canc} 
For any elements $x,y,z$ in an MV-algebra, if $x\oplus z=y\oplus z$ and $x\odot z=0=y\odot z$, then $x=y$. In particular, if $x\odot y=0$, then $y=x\oplus y$ implies $x=0$.\qed
\end{lemma}
Finally, we add a technical lemma for later use.
\begin{lemma}[{\cite[1.6.2]{cdm2000}}]\label{mv4} 
The following equations hold in any MV-algebra.
\begin{enumerate}
\item $x\oplus y\oplus (x\odot y)=x\oplus y$.
\item $(x\ominus y)\oplus ((x\oplus \neg y)\odot y)=x$. \qed
\end{enumerate}
\end{lemma}
\subsection{Ideals, the radical, and semisimplicity.}\label{ss:ideals}
A subset of an MV-algebra $A$ is an \emph{ideal} of $A$ provided that it contains $0$, is downward-closed (in the underlying lattice order of $A$), and is closed under  $\oplus$. Ideals of $A$ are precisely the kernels (=inverse images of $0$) of MV-homomorphisms with domain $A$. Moreover, ideals of $A$ are in one-one correspondence with congruence relations on $A$, see \cite[1.2.6]{cdm2000}. Specifically, given an ideal $I\subseteq A$, define the binary relation
\begin{align*}
x\equiv_{I} y \text{ if, and only if, } \d(x,y) \in I.
\end{align*}
Then $\equiv_{I}$ is a congruence on $A$, and we write $A/I$ to denote the quotient set with its natural MV-algebraic structure. Conversely, any congruence relation $\equiv$ on $A$ determines the ideal $I_{\equiv}:=\{x\in A \mid x\equiv 0\}$. 
The usual homomorphism theorems hold, cf.\ e.g.\ \cite[1.2.8]{cdm2000}. 
Those ideals of $A$ which are proper (i.e.\ not equal to $A$) and maximal with respect to  set-theoretic inclusion are called \emph{maximal ideals}. We set
\begin{align*}
\Max{A}:=\{\m\subseteq A\mid \m \text{ is a maximal ideal of } A\}.
\end{align*}
As a consequence of Zorn's Lemma, we have the important

\begin{lemma}[{\cite[1.2.15]{cdm2000}}]\label{lindenbaum}Every non-trivial MV-algebra has a maximal ideal.
\qed\end{lemma}
Thus,  an MV-algebra $A$ is trivial (=its underlying set is a singleton), or equivalently, is the terminal object in $\MV$, if, and only if, $\Max{A}=\emptyset$.
\begin{notation}\label{n:iterated sum}Given an element $x$ of an MV-algebra, and given $n\in\N$, we use the shorthand 
\begin{align*}
nx:=\underbrace{x\oplus\cdots\oplus x}_{n {\text{ times}}}.
\end{align*}
\end{notation}
The \emph{radical} ideal $\rad{A}$ is  the intersection of all maximal ideals of $A$, in symbols,
\begin{align*}
\rad{A}:=\bigcap{\Max{A}}.
\end{align*} 
 A non-zero element $x\in A$ is \emph{infinitesimal} if $nx\leq\neg x$ for each $n\in\N$. The radical ideal is precisely the collection of all infinitesimal elements, along with zero:
\begin{lemma}[{\cite[3.6.4]{cdm2000}}] \label{mvrad}
For any MV-algebra $A$, $\rad{A}=\{x\in A\mid  nx\leq \neg x \; \; \mbox{for all} \; \; n\in\N\}$.\qed
\end{lemma}
Because of  the MV-cancellation law (Lemma \ref{canc}), the next lemma implies that the operation $\oplus$ is cancellative on sufficiently small (in particular, infinitesimal) elements:
\begin{lemma} \label{mv5}
If $A$ is an MV-algebra and $x,y\in A$ are such that $x\odot x=0$ and $y\odot y=0$, then $x\odot y=0$. In particular, if $x,y \in \rad{A}$, then $x\odot y=0$.
\end{lemma} 
\begin{proof}To prove  the first assertion,
 by Lemmas \ref{mv4}.(1) and \ref{canc} it suffices to prove $(x\oplus y)\odot(x\odot y)=0$. Notice that \begin{align*}
(x\oplus y)\odot(x\odot y)&=(x\oplus y)\odot \neg(\neg x \oplus \neg y) \\
&=\neg(\neg(x\oplus y)\oplus (\neg x\oplus \neg y)). \end{align*}
This means that \begin{align*}(x\oplus y)\odot(x\odot y)=0 &\Leftrightarrow \\
\neg(x\oplus y)\oplus (\neg x\oplus \neg y)=1 &\Leftrightarrow \\
x\oplus y\leq \neg x\oplus\neg y\tag{Lemma \ref{mv1}}&. \end{align*}
To prove the latter inequality, just note that $x\odot x=0$ is equivalent to $x\leq \neg x$ by Lemma \ref{mv1}, and then apply  Lemma \ref{mv2}.  

The second assertion follows at once from the first upon applying Lemma \ref{mvrad}.
\end{proof}
An MV-algebra $A$ is \emph{simple} if it has no non-trivial proper ideals, and  \emph{semisimple} if $\rad{A}=\{0\}$, i.e.\ $A$ is free of infinitesimal elements. Equivalently, semisimple MV-algebras are precisely the  
subdirect products (\cite[II.\S8]{Burris:81}, cf.\ also Footnote \ref{fn:sub-direct}) of simple MV-algebras. Simple and semisimple MV-algebras are central to this paper. We shall see in Lemma \ref{MVHolder}  (H\"older's Theorem) that simple MV-algebras can  
be uniquely identified with subalgebras of $\int$; in Lemma \ref{MVmax} that $\Max{A}$ can be equipped with a compact Hausdorff topology induced by $A$; and in Corollary \ref{cor:reflection}.(2) that if $A$ is semisimple, it can be canonically identified with a subalgebra of $\C(\Max{A})$.
\subsection{The functor $\Gamma$, and lattice-groups.}\label{ss:gamma}
We assume familiarity with the basic notions on lattice-ordered groups. For background  see e.g.\ \cite{BigardKeimel77, Darnel, Glass}. For background on the functor $\Gamma$ see \cite[Chapters 2 and 7]{cdm2000}. 
We recall the needed definitions. By an \emph{$\ell$-group} we throughout mean an Abelian group $G$, always written additively,
with a compatible lattice structure --- that is, the group addition distributes over  the lattice operations. The $\ell$-group is \emph{unital} if it is equipped with a distinguished element $u\in G$ 
such that for all $g \in G$ there is $n\in \N$ 
with $nu\geq g$. We write $\Grp$ for the category whose objects are unital $\ell$-groups, and whose morphisms are the \emph{unital $\ell$-homomorphisms}, i.e.\ the lattice and group homomorphisms that preserve the units. Subobjects in 
$\Grp$, i.e.\ sublattice subgroups that contain the unit, are called \emph{unital $\ell$-subgroups}.

To each unital $\ell$-group $(G,u)$ we associate its \emph{unit interval}, namely
$\Gamma(G,u):=\{g\in G \mid 0 \leq g \leq u\}$, and we endow it with the operations 
\begin{align*}
x\oplus y := (x+y)\wedge u, \\
\neg x := u-x.
\end{align*}
Then $(\Gamma(G,u),\oplus,\neg,0)$ is an MV-algebra \cite[2.1.2]{cdm2000}.  Thus, in the additive group $\R$ under natural order, $1 \in \R$ is a unit and $\Gamma(\R,1)=\int$ is the standard MV-algebra, cf.\ Example \ref{ex:c(x)mv}. 
Similarly, $\Gamma(\Z,1)=\2$, the two-element Boolean algebra.

 Each unital $\ell$-homomorphism $(G,u)\to (G',u')$ restricts to an MV-homomorphism $\Gamma(G,u)\to\Gamma(G',u')$. Hence we obtain a functor $\Gamma\colon \Grp\to\MV$. Note that we will at times write $\Gamma(G)$ for $\Gamma(G,u)$ when $u$ is understood.
 
A \emph{good sequence} $(a_{i})$ in an MV-algebra $A$ is a sequence $a_1,a_2,\ldots$ of elements of $A$ such that  $a_i\oplus a_{i+1}=a_i$ for each $i\in\N$, and there is $n_0\in\N$ such that $a_n=0$ for 
every $n\geq n_0$. For example, good sequences in Boolean algebras are exactly the non-increasing, eventually-zero sequences of elements. Indeed, $\vee$ and $\oplus$ coincide in Boolean algebras.

The collection $M_{A}$ of all good sequences of $A$ is a cancellative Abelian monoid, when addition of good sequences $(a_{i})$ and $(b_{i})$ is defined by
\[
(a_{i})+(b_{i}):= \left(\,a_i\oplus (a_{i-1}\odot b_1)\oplus\cdots\oplus(a_1\odot b_{i-1})\oplus b_i\,\right).
\]
See \cite[2.3.1]{cdm2000}. Further, $M_{A}$ can be lattice-ordered by declaring $(a_{i})\leq (b_{i})$ if, and only if, $a_{i}\leq b_{i}$ for each $i$. Let now $M_A \to \Xi(A)$ denote the universal arrow from the monoid $M_{A}$ to its universal enveloping (Grothendieck) group. 
Since $M_{A}$ is cancellative, the arrow is a monomorphism; let us 
identify $M_{A}$ with its image in  $\Xi(A)$. 
The order of $M_{A}$ extends uniquely to a lattice order on $\Xi(A)$ that makes it into an $\ell$-group \cite[2.4.2]{cdm2000}. The good sequence (1,0,0,\ldots) is a unit of $\Xi(A)$.
Now consider a homomorphism of MV-algebras $h\colon A\to B$, and define a function $h^{*}\colon M_{A}\to M_{B}$ by setting $h^{*}(\,(a_{i})\,):=(h(a_{i}))$. Then $h^{*}$ extends uniquely to a 
unital $\ell$-homomorphism $\Xi(h)\colon\Xi(A)\to\Xi(B)$, \cite[pp.\ 139--140]{cdm2000}. This defines  a functor $\Xi\colon\MV\to \Grp$.
\begin{lemma}[Mundici's equivalence {\cite[7.1.2, 7.1.7]{cdm2000}}]\label{l:gamma}
The functors $\Gamma\colon\Grp\to\MV$ and $\Xi\colon \MV\to\Grp$ yield an equivalence of categories.\qed
\end{lemma}

\subsection{The relationship with $\mathrm{C}^*$-algebras.}\label{ss:Cstar}
In this subsection we assume familiarity with $\mathrm{C}^*$-algebras. For the needed background see e.g.\ \cite{Dixmier77}.
We write $\Cst$ for the category of complex commutative unital $\mathrm{C}^*$-algebras. 
If $A$ is  an algebra in $\Cst$, write $\H(A)$ for the set of self-adjoint elements of $A$. The latter is a real linear space complete in the norm of $A$. 
Moreover, $\H(A)$ carries a partial order: an element $x\in\H(A)$ is \emph{positive} if there is $y\in\H(A)$ such that $x=y^2$. This partial order is in 
fact a lattice order, and this is known to be characteristic of commutative $\mathrm{C}^*$-algebras (see \cite{Sherman:1951}). Since the addition of 
$A$
is compatible with this lattice order, $\H(A)$ is an $\ell$-group with unit $1_A$, where the latter denotes the (ring) unit of $A$. The correspondence
$A\mapsto\H(A)$ is functorial, and yields a functor $\H\colon\Cst\to\Grp$. It can be shown that $\H$ is full and faithful, and moreover admits a left
adjoint $\mathrm{L}\colon\Grp\to\Cst$, so that $\Cst$ is equivalent to a full reflective subcategory of $\Grp$.  For a unital $\ell$-group $(G,u)$, $\mathrm{L}(G,u)$ can 
be described as $\C(\Max{G}, \CC)$, the $\mathrm{C}^*$-algebra of complex-valued continuous functions on the (compact Hausdorff) space of 
maximal congruences of $G$. The space $\Max{G}$ is canonically homeomorphic to the space $\Max{\Gamma(G,u)}$ (see Subsections \ref{ss:ideals} and \ref{ss:gamma}).
If $X$ is a compact Hausdorff space, define
\begin{align*}
\C(X,\R):=\{f\colon X\to\R\mid f \,\text{ is continuous}\}.
\end{align*}
Since $\R$ is an $\ell$-group under addition and natural order, with unit $1$, we regard $\C(X,\R)$ as an object of $\Grp$ by defining operations pointwise; the unit of $\C(X,\R)$ is thus $1_X\colon X\to \R$, the function constantly equal to $1$ on $X$. 
Now we see that $\H\circ \mathrm{L}(G,u)=\C(\Max{G},\R)$.
Observe\footnote{Here and throughout, we abuse notation  and 
consider $\C(X)$ as a subset of $\C(X,\R)$. That is, we tacitly identify functions $X\to\R$ whose range is contained in $\int$ with functions $X\to \int$.} that $\Gamma(\C(\Max{G},\R),1_{X})=\C(\Max{G})$, the MV-algebra \eqref{eq:c(x)} of Example \ref{ex:c(x)mv}.
The following three objects determine each other up to natural isomorphism: the $\mathrm{C}^*$-algebra $A$, the unital $\ell$-group $\H(A)$, and the MV-algebra $\Gamma(\H(A))$.
For instance, consider the $\mathrm{C}^*$-algebra $\CC$ of complex numbers. The self-adjoint elements of $\CC$ form the $\ell$-group  $\R$ with unit $1$, whose unit interval is $\Gamma(\R,1)=\int$. The unital  $\ell$-group $(\R,1)$ uniquely determines $\CC$, by complexification (see e.g.\ \cite[pp.\ 71--80]{Goodearl82}). The MV-algebra $\int$ uniquely determines $(\R,1)$, by Mundici's theorem (Lemma \ref{l:gamma}),  and  therefore also the $\mathrm{C}^*$-algebra $\CC$.
 Returning to the $\mathrm{C}^*$-algebra $A$, note that
  $\Gamma(\H(A),1_A)$ is the collection of all its non-negative, self-adjoint elements whose norm does not exceed $1$ (i.e., that lie in  the unit ball of $A$). Summing up:  
 \emph{any $\mathrm{C}^*$-algebra in $\Cst$ is uniquely 
 determined by the MV-algebra of non-negative self-adjoint elements in its unit ball.}
 
 If we compose the functor $\Gamma\circ \H$ with the underlying-set functor $\MV\to\Set$, we obtain a functor $P\colon\Cst\to\Set$ which is known to be monadic. This is one way of stating that $\Cst$ is equivalent to a variety of algebras. In fact, there are several results in the literature concerned with the monadicity of functors from categories of $\mathrm{C}^*$-algebras 
 to $\Set$. Although none of these results is  needed in this paper, it is appropriate to mention the main ones. In 1971, Negrepontis proved \cite[Theorem 1.7]{Negrepontis71} that the \emph{unit ball functor}, sending an 
 algebra in $\Cst$ to the set of its elements 
 whose norm does not exceed $1$, is monadic. Isbell's cited paper \cite{Isbell} entails the analogous fact for the \emph{Hermitian unit ball 
 functor} that maps an object of $\Cst$ to 
 the set of self-adjoint elements in its unit ball. In 1984, Van Osdol \cite[Theorem 1]{vanOsdol84} generalised Negrepontis' result by proving that, 
 in fact, the category of complex (not necessarily commutative, nor necessarily unital) $\mathrm{C}^*$-algebras is monadic over $\Set$ with respect to 
 the unit ball functor. Further monadicity results were provided in \cite{PelletierRosicky89, PelletierRosicky93} for various functors, including the \emph{positive unit ball functor} that sends a (not necessarily commutative)
 $\mathrm{C}^*$-algebra to the set of non-negative, self-adjoint elements in its unit ball: in the commutative setting this is precisely the functor $P$ above.
 Our main result provides a finite equational axiomatisation of those algebraic structures consisting of positive elements in unit balls of $\mathrm{C}^*$-algebras in $\Cst$.
 We take advantage of the fact that these can be conceived of as MV-algebras with an additional infinitary operation.
\begin{remark}\label{rm:AF}
We alert the reader to the fact that there is a second, different way of relating $\mathrm{C}^*$-algebras to MV-algebras. Bratteli's approximately finite-dimensional (AF) $\mathrm{C}^*$-algebras are classified by  certain unital partially ordered Abelian groups, namely Elliott's ordered $K_0$. The partially ordered groups arising in this manner are known as dimension groups, and include all unital $\ell$-groups. Thus, for each AF $\mathrm{C}^*$-algebra $A$ with lattice-ordered $K_0$ there is an associated MV-algebra, unique up to isomorphism, namely $\Gamma\circ K_0(A)$, and all MV-algebras arise in this manner. Those AF $\mathrm{C}^*$-algebras which are commutative are precisely those of the form $\C(X,\CC)$ for $X$ a Stone space. 
For AF $\mathrm{C}^*$-algebras, dimension groups, and their relationship with MV-algebras the reader can consult e.g.\ \cite{RLL00, Goodearl86, Mundici86}.
\end{remark}

\section{The Cignoli-Dubuc-Mundici adjunction.}\label{s:reflection}
It was first observed in \cite[Section 4]{Cignolietal2004} that there is an adjunction between the category $\MV$ of MV-algebras and $\KH^{\rm op}$. 
Following \cite{PT91}, this can be regarded as a dual adjunction between the concrete categories $\MV$ and $\KH$. Here, as usual, \emph{concrete} means \emph{equipped with a faithful functor to $\Set$}; both $\MV$ and $\KH$ are equipped with their underlying-set functors. Moreover, these underlying-set functors are represented by the  MV-algebra freely generated by one element, and by the one-point space, respectively. In this situation one can ask whether the  adjunction at hand is  \emph{natural} in the sense of \cite[p.\ 116]{PT91}. Indeed, the Cignoli-Dubuc-Mundici  adjunction is natural, and is induced by the unit interval $\int$ conceived of as a \emph{dualising object} \cite[Definition 1.6]{PT91}: recall from Example \ref{ex:c(x)mv} that the compact Hausdorff space $\int$ may be regarded  as an MV-algebra. Conversely, the mere knowledge of the fact that $\int$ is a dualising object for $\MV$ and $\KH$ would yield the Cignoli-Dubuc-Mundici adjunction from \cite[Theorem 1.7]{PT91}. In fact, it is known that, under relatively mild conditions, \emph{any} dual adjunction between concrete categories is necessarily induced by a dualising object; see \cite{LR79, DT89, PT91}. In short, general categorical results in duality theory apply to our specific situation. In this section we give an account of the Cignoli-Dubuc-Mundici adjunction that  indicates the way the result fits into general duality theory, and at the same time explains how to relate the treatment based on representable functors with the classical approach that uses maximal ideals and the Stone topology.

Given a compact Hausdorff space $X$, recall the MV-algebra $\C(X)$ defined in Example \ref{ex:c(x)mv}. Given a continuous map $f\colon X \to Y$ in $\KH$, it is elementary that the induced function 
\begin{align*}
\C(f)\colon \C(Y)&\longrightarrow\C(X),\\
g\in \C(Y)&\longmapsto g\circ f \in \C(X)
\end{align*}
is a morphism in \MV. We therefore regard $\C$ as a functor:
\begin{align*}
\C \colon  \KH^{\rm op}\to  \MV. 
\end{align*}
Observe that, in the terminology of \cite{PT91} and by virtue of the continuity of the MV-algebra operations of $\int$, this functor $\C$ is a lifting along the underlying-set functor $\MV\to \Set$ of the contravariant\footnote{The reader is advised to bear in mind that while we only use covariant functors, \cite{PT91} uses contravariant functors.} $\hom$-functor $\hom_{\KH}{(-,\int)}$. Symmetrically, regarding $\int$ as an MV-algebra, we consider the contravariant $\hom$-functor $\hom_{\MV}{(-,\int)}\colon\MV\to \Set$.
Although this is not obvious, it turns out that this functor admits a lifting to $\KH$ along the  underlying-set functor $\KH\to \Set$, which we again write
\begin{align}\label{eq:homfunctor}
\hom_{\MV}{(-,\int)}\colon\MV\to \KH.
\end{align}
The lifting is achieved by equipping $\hom_{\MV}{(A,\int)}$ with the initial topology with respect to the family of evaluations $\left(\varphi_{a}\right)_{a \in A}$, where each evaluation  sends a homomorphism $h\colon A\to\int$ to $h(a)$.  In classical approaches the functor  \eqref{eq:homfunctor} is equivalently replaced by the maximal spectrum functor, as we now explain. For  an arbitrary MV-algebra $A$ and any subset $S\subseteq A$, define 
\begin{align*} 
\V{(S)}:=\{\m \in \Max{A}\mid S\subseteq \m\}.
\end{align*}
 If $a \in A$, write $\V{(a)}$ as a shorthand for $\V{(\{a\})}$.  Then:
\begin{lemma}[{\cite[4.15, and Section 4.4]{Mundici11}}] \label{MVmax}
For any MV-algebra $A$, the collection 
 \begin{align*} 
 \{\V{(I)}\mid I \text{ is an ideal of } A\}
 \end{align*}
 is the set of closed sets for a compact Hausdorff topology on $\Max{A}$.  The collection
 \begin{align*} 
 \{\V{(a)}\mid  a\in A\}
 \end{align*}
is  a basis of closed sets for this topology.\qed
\end{lemma}
The topology in Lemma \ref{MVmax} is the \emph{Stone topology} of $A$, also known as  the \emph{Zariski} or \emph{hull-kernel} topology. The set $\Max{A}$ equipped with the Stone topology is known as the \emph{maximal spectrum} or the \emph{maximal spectral space} of $A$. 
\begin{lemma}\label{mv6} 
For any MV-homomorphism $h\colon A\to B$, if $\m\in\Max{B}$ then $h^{-1}(\m)\in\Max{A}$. Moreover, the inverse-image map $h^{-1}\colon \Max{B}\to\Max{A}$ is continuous with respect to the Stone topology.
\end{lemma}
\begin{proof}The first assertion is proved in \cite[1.2.16]{cdm2000}. The second assertion is a straightforward verification using Lemma \ref{MVmax}.
\end{proof}
In light of Lemma \ref{mv6} we henceforth regard $\Max$ as a functor:
\begin{align}\label{eq:maxfunctor}
\Max \colon \MV \to \KH^{\rm op}.
\end{align}
The functor \eqref{eq:maxfunctor} and the (covariant version of the) functor \eqref{eq:homfunctor} are naturally isomorphic, and we will freely  identify them  in the rest of this section whenever convenient. This identification rests on the following key fact.
\begin{lemma}[H\"older's Theorem ({\cite{Holder01}, see also \cite{HolderTR1,HolderTR2})} for MV-algebras {\cite[3.5.1]{cdm2000}}]\label{MVHolder}
For every MV-algebra $A$, and for every  $\m\in\Max{A}$,  there is  \emph{exactly one}  homomorphism of MV-algebras \[\h_{\m}\colon\tfrac{A}{\m}\to\int,\]
and this homomorphism is injective.
Moreover, the correspondence  $\m \mapsto \h_{\m}$ is a bijection between $\Max{A}$ and $\hom_{\MV}{(A,\int)}$.
\qed
\end{lemma}
\begin{remark}\label{r:notsurjective} We emphasise that H\"older's map $\h_{\m}$ in Lemma \ref{MVHolder} need not be surjective, in contrast to its classical functional-analytic counterparts by Stone   \cite[p.\ 85]{StoneSpectra2} (for the real numbers) and 
Gelfand \cite[Satz 7]{gelfand1941} (for the complex numbers).
\end{remark}
One can now verify that there is an adjunction  $\Max\dashv\C\colon \KH^{\rm op}\to \MV$, or equivalently,  that the contravariant functors $\hom_{\MV}{(-,\int)}$ and $\hom_{\KH}{(-,\int)}$ yield a dual adjunction between $\MV$ and $\KH$, in the sense of \cite{PT91}. In fact, the existence of the latter dual adjunction follows from \cite[Theorem 1.7]{PT91} (see also \cite[Proposition 2.4]{DT89}). The unit $\eta$ and the counit $\epsilon$ of the adjunction are necessarily given by evaluation as in \cite[(6) and (7) on p.\ 115]{PT91}. 
In detail, if $a$ is an element of the MV-algebra $A$, in light of Lemma \ref{MVHolder} evaluation at $\m\in\Max{A}$ yields 
\begin{align*}
\widehat{a}\colon \Max{A}&\longrightarrow \int \\
\m &\longmapsto \h_{\m}(\tfrac{a}{\m}).
\end{align*}
Evaluation thus provides us\footnote{We point out that if one directly uses the classical functor $\Max$, without identifying it with $\hom_{\MV}{(-,\int)}$, one needs to prove that $\widehat{a}\colon\Max{A} \to \int$ is continuous with respect to the Stone topology on the domain and the Euclidean topology on the codomain.  See \cite[4.16.iii]{Mundici11} for a proof.} with a homomorphism
\begin{align}
\eta_{A} \colon A &\longrightarrow \C(\Max{A})\label{eq:eta}\\
a&\longmapsto \widehat{a}\nonumber
\end{align}
for each MV-algebra $A$.

Conversely, using maximal ideals, for any space $X$ in $\KH$ there is a continuous map
\begin{align}
\epsilon_{X} \colon X &\longrightarrow \Max{\C(X)}\label{eq:epsilon}\\
x&\longmapsto \{f\in\C(X)\mid f(x)=0\}.\nonumber
\end{align}
(Compare \cite[4.16]{Mundici11}.)
Writing $\id{{\sf C}}$ for the identity functor on a category {\sf C}, we  can summarise as follows.
\begin{theorem}[{\cite[Propositions 4.1 and 4.2]{Cignolietal2004}}]\label{t:Max-C-adj}The functor $\Max$ is left adjoint to $\C$, i.e.\ $\Max\dashv\C\colon \KH^{\rm op}\to \MV$. The unit and the counit of the adjunction are the natural transformations $\eta \colon \id{\MV} \to \C\circ\Max$ and $\epsilon \colon \Max\circ\C \to \id{\KH^{\rm op}}$ whose components are given by \eqref{eq:eta} and \eqref{eq:epsilon}, respectively.\qed
\end{theorem}
If $F\dashv G\colon {\sf D}\to{\sf C}$ is any adjunction with unit and counit $\eta$ and $\epsilon$, respectively, write $\fix{\eta}$ for the full subcategory of ${\sf C}$ on the objects \emph{fixed by $\eta$}, that is,
\[
\{O \mid O \text{ is an object of } {\sf C} \text{ and } \eta_{O} \text{ is an isomorphism}\};
\]
similarly, write $\fix{\epsilon}$  for the full subcategory of ${\sf D}$ on the objects fixed by $\epsilon$.
It is then standard  that the adjunction restricts to an equivalence  between $\fix{\eta}$ and $\fix{\epsilon}$. In our situation, the adjunction $\Max\dashv \C$ actually fixes all compact Hausdorff spaces, because of the following fact {\it \`a la} Stone-Kolmogorov (\cite[Theorem 2]{st2},\cite[Theorem I]{GK39}): 
\begin{lemma}[{\cite[4.16]{Mundici11}}]\label{l:epsiloniso}For each space $X$ in $\KH$, the component $\epsilon_{X}\colon X \to \Max{\C(X)}$ of $\epsilon$ at $X$ given by \eqref{eq:epsilon} is a homeomorphism. 
\qed\end{lemma}
By contrast, the main objective of this paper amounts to an equational characterisation of $\fix{\eta}$. Let us first adopt a notation for $\fix{\eta}$ that is more perspicuous in our situation.
\begin{notation}\label{not:mvc}
Let us agree to write $\MV_{\C}$ for the full subcategory of $\MV$ on those objects fixed by $\eta$.
\end{notation}
\begin{corollary}\label{cor:reflection}The following hold.
\begin{enumerate}
\item The adjunction $\Max \dashv \C$ restricts to an equivalence between   $\MV_{\C}$ and $\KH^{\rm op}$.
\item $\MV_{\C}$ is a reflective subcategory of $\MV$, with reflection $\eta$. Moreover, the component $\eta_{A}\colon A\to\C(\Max{A})$ of the reflection is injective if, and only if, $A$ is semisimple.
\end{enumerate} 
\end{corollary}
\begin{proof} By Lemma \ref{l:epsiloniso}, the counit $\epsilon$ fixes all spaces in $\KH^{\rm op}$. On the other hand, the unit $\eta$ fixes precisely $\MV_{\C}$ by the definition of the latter. This proves the first 
item. The proof of the first statement in the second item is an easy consequence of  Lemma \ref{l:epsiloniso}, cf.\ \cite[Lemma 2.11]{PT91}. The second statement in the second item is proved in \cite[4.16.iii]{Mundici11}. (Let us note that with Lemma \ref{MVHolder} available, the semisimplicity of $A$ is equivalent to  the fact that the cone of all homomorphisms $A\to \int$ is a monosource, i.e.\ is jointly monic, and then the statement at hand follows from \cite[Lemma 2.3]{PT91}.)
\end{proof}
\begin{remark}\label{r:factorisation}The reflection in Corollary \ref{cor:reflection}.(2) factors into three further reflections that clarify the gap 
between a general MV-algebra and one in the subcategory $\MV_{\C}$. We describe this factorisation here, omitting proofs and details.   Pick an MV-algebra $A$ in $\MV$. The natural quotient 
map $$A \longrightarrow \tfrac{A}{\rad{A}}$$ is the universal homomorphism from $A$ to a semisimple MV-algebra. The associated reflection is to be 
thought of as annihilation of infinitesimals. There is a further universal arrow 
$$\tfrac{A}{\rad{A}}\longrightarrow \left(\tfrac{A}{\rad{A}}\right)^{\rm d}$$ that embeds a semisimple MV-algebra into its divisible hull. The 
divisible hull can be 
defined, for any MV-algebra $B$, as $\Gamma(\Xi(B)^{\rm d})$, where  $\Xi(B)^{\rm d}$ is the classical divisible hull of a lattice-ordered Abelian 
group. It is also classical that the unit of any unital semisimple Abelian lattice-group determines a norm on it.\footnote{Divisibility and 
completeness in the unit norm  of lattice-groups are already used by Stone in \cite{st2}. Yosida \cite{Yosida41} 
only uses completeness in the unit norm, because his setting is that of vector lattices (=lattice-ordered real linear spaces) which, as lattice-groups, are 
automatically divisible. It is useful to compare the Stone-Yosida approach, from which ours descends, to the one by Kakutani 
\cite{Kakutani40, Kakutani41} and by the Krein brothers \cite{KreinKrein40}. In the former approach, the norm is derived from the algebraic structure; in the latter,  the 
norm is primitive.}  Hence any semisimple 
MV-algebra carries a \emph{unit norm} uniquely determined by its algebraic 
structure, and there is a universal arrow 
$$\left(\tfrac{A}{\rad{A}}\right)^{\rm d}\longrightarrow \left(\left(\tfrac{A}{\rad{A}}\right)^{\rm d}\right)^{\rm c}$$
given by Cauchy completion in the unit norm. Finally, the composition
$$A\longrightarrow \left(\left(\tfrac{A}{\rad{A}}\right)^{\rm d}\right)^{\rm c}$$
amounts to the homomorphism $\eta_A\colon A\to \C(\Max{A})$ as defined in \eqref{eq:eta}, because there is a natural isomorphism $\C(\Max{A})\cong \left(\left(\frac{A}{\rad{A}}\right)^{\rm d}\right)^{\rm c}$. The existence of this isomorphism is due to the Stone-Weierstrass Theorem for MV-algebras, to be proved in 
Lemma \ref{l:mvsw}, together with the  observation that there are  homeomorphisms $\Max{A}\cong \Max{\frac{A}{\rad{A}}}\cong  \Max{\left(\frac{A}{\rad{A}}\right)^{\rm d}}$.
\end{remark}
\section{Elementary results on $\delta$-algebras.}\label{s:delta}
We now introduce $\delta$-algebras, a  notion that will afford an equational characterisation of the full subcategory $\MV_{\C}$  of $\MV$ defined in the previous section.
To this aim we consider operation symbols $\delta,\oplus,\neg,0$, where $\delta$ has arity $\aleph_0$, $\oplus$ is binary, $\neg$ is unary, and $0$ is a constant.
\begin{notation}
The operation $\delta$ takes as argument a countable sequence of terms. We write $\vec{x},\vec{y}$ and $\vec{0}$ as a shorthand for $x_1,x_2,\ldots$, $y_1,y_2,\ldots$, and $0,0,\ldots$, respectively. 
\end{notation}
Let us define a unary operation $\fm$ by setting \[\fm(x):= \delta(x,\vec{0}).\]
In the definition below, the operations $\d$ and $\ominus$ denote Chang's distance and truncated subtraction on an MV-algebra, respectively. Please see Section \ref{s:pre} for details.
\begin{definition}\label{d:delta}
 A \emph{$\delta$-algebra} is an algebra $(A,\delta,\oplus,\neg,0)$ such that $(A,\oplus,\neg,0)$ is an MV-algebra, and the following identities are satisfied. 
\begin{description}
 \item[A1] $\d\left(\delta(\vec{x}),\delta(x_1,\vec{0})\right)=\delta(0,x_2,x_3,\ldots)$.
 \item[A2] $\fm(\delta(\vec{x}))=\delta(\fm(x_1),\fm(x_2),\ldots)$.
 \item[A3] $\delta\left(x,x,\ldots\right)=x$.
 \item[A4] $\delta(0,\vec{x})=\fm(\delta(\vec{x}))$.
 \item[A5] $\delta(x_1\oplus y_1,x_2\oplus y_2,\ldots)\geq \delta(x_1,x_2,\ldots)$.
 \item[A6] $\fm(x\ominus y)=\fm(x)\ominus\fm(y)$.
\end{description}
\end{definition}
\begin{remark}
Since every MV-algebra has a lattice reduct (see Lemma \ref{mv1}), Axiom \textbf{A5} can be  written in  equational form in the obvious manner.
\end{remark}
A homomorphism $h\colon A\to B$ of $\delta$-algebras is a homomorphism of the underlying MV-algebras that also preserves $\delta$. The latter condition means, as usual, that for all $\{x_i\}_{i\in\N}\subseteq A$,
\[
h(\delta(\vec{x}))=\delta(h(x_1),h(x_2),\ldots).
\] 
Homomorphisms of $\delta$-algebras will be called  \emph{$\delta$-homomorphisms}.
\begin{definition}\label{d:deltacat}
We let $\Va$ denote the category whose objects are $\delta$-algebras and whose morphisms are $\delta$-homomorphisms.
\end{definition}
\begin{remark}\label{r:reduct}There is a forgetful functor $\forg{-} \colon \Va\to \MV$ that sends a $\delta$-algebra to its underlying MV-algebra. There is no obvious reason why this functor should be full. We will eventually show that this is the case in Theorem \ref{thm:fullness}.
\end{remark}
Recall from Example \ref{ex:c(x)mv} the standard example of MV-algebra, $\int$. It is an exercise to check that the MV-algebra $\int$ becomes a $\delta$-algebra when equipped with   the additional operation $\delta(u_{1},u_{2},\ldots):=\sum_{i=1}^{\infty}\frac{u_i}{2^i}$. Since this operation is continuous, we obtain:
\begin{lemma}\label{CofX}
 Given any compact Hausdorff space $X\neq \emptyset$, the MV-algebra $\C(X)$ is a $\delta$-algebra if, for all $\{g_i\}_{i\in\N}\subseteq \C(X)$,  $\delta$ is defined as \[\delta(g_{1},g_{2},\ldots):=\sum_{i=1}^{\infty}\frac{g_i}{2^i}.\]
\end{lemma}
\begin{proof}
 Straightforward verification.
\end{proof}
\noindent (Note that if $X=\emptyset$ then $\C(X)$ is a singleton, and thus has a unique structure of $\delta$-algebra.)
\begin{remark}\label{r:struct}For the rest of this paper, whenever we regard $\C(X)$ as a $\delta$-algebra, for $X$ a compact Hausdorff space, we always assume that the operation $\delta$ is the one provided by Lemma \ref{CofX}. (We shall eventually prove in Corollary \ref{cor:unique}, that there actually is no other $\delta$-algebraic structure expanding the MV-algebra $\C(X)$.) 
\end{remark}
We now embark on the proof of several elementary algebraic results on $\delta$-algebras. These  provide the ground on which the representation theory of the later sections rests.

Recall the unary derived operation $\fm$. We  define a further unary operation $\fmn$ for each   $n\in\N$ by iterating $\fm$, namely: 
\[\fmn(x):=\underbrace{\fm(\cdots(\fm}_{n \, \text{times}}(x))\cdots).\]
\begin{lemma}\label{lem:1}
 The following equations hold in an arbitrary $\delta$-algebra, for each $n\in\N$.
 \begin{enumerate}
 \item $\fmn(\delta(\vec{x}))=\delta(\fmn(x_1),\fmn(x_2),\ldots)$.
 \item  $\fmn(\delta(\vec{x}))=\delta(\underbrace{0,\ldots,0}_{n \, \text{times}},\vec{x})$.
 \item $\delta(\vec{x})\geq\delta(x_1,\ldots,x_n,\vec{0})$.
 \item $\delta(\vec{x})=\delta(x_1,\vec{0})\oplus\delta(0,x_2,x_3,\ldots)$.
  \item $\fmn(x\ominus y)=\fmn(x)\ominus\fmn(y)$.
  \item $\neg\fm(1)=\fm(1)$.
   \item $\fmn(x)\leq x$.
 \item If $x\leq y$ then $\fmn(x)\leq \fmn(y)$.
 \item $\fm(x)\odot\fm(x)=0$.
 \end{enumerate}
\end{lemma}
\begin{proof} Items 1--2 are proved by a straightforward induction on $n$ using Axioms \textbf{A2} and \textbf{A4}, respectively.  
Item 3 follows at once from  \textbf{A5} upon defining $y_1:=0,\ldots,y_n:=0$, and $y_i:=x_i$ for all $i>n$.  For item 4, note $\delta(\vec{x})\geq\delta(x_1,\vec{0})$ by item 3. Hence, \begin{align*}
\delta(\vec{x})&=\delta(x_1,\vec{0})\oplus \d\left(\delta(\vec{x}),\delta(x_1,\vec{0})\right) \tag{Lemma \ref{mv3}}\\
&=\delta(x_1,\vec{0})\oplus \delta(0,x_2,x_3,\ldots). \tag{\textbf{A1}} \end{align*}

Item 5 is Axiom \textbf{A6} if $n=1$. If $n>1$, \begin{align*}
\fmn(x\ominus y)&=\fm(\fmnn(x\ominus y)) \\
&=\fm(\fmnn(x)\ominus \fmnn(y)) \tag{inductive hypothesis} \\
&=\fm(\fmnn(x))\ominus \fm(\fmnn(y)) \tag{\textbf{A6}} \\
&=\fmn(x)\ominus \fmn(y).
\end{align*}

For item 6, recall that $1:= \neg 0$, and that the equation $1\ominus x=\neg x$  holds in any MV-algebra. Hence, 
\begin{align*}
\neg\fm(1)&=1\ominus\fm(1) \\
&=\delta(1,1,\ldots)\ominus\delta(1,\vec{0}) \tag{\textbf{A3}} \\
&=\d\left(\delta(1,1,\ldots),\delta(1,\vec{0})\right) \tag{Lemmas \ref{lem:1}.(3), \ref{mv1}} \\
&=\delta(0,1,1,\ldots) \tag{\textbf{A1}} \\
&=\fm(\delta(1,1,\ldots)) \tag{\textbf{A4}} \\
&=\fm(1). \tag{\textbf{A3}}                                                                                                                    
\end{align*}

For item 7, note that
 \[\fm(x)=\delta(x,\vec{0})\leq \delta(x,x,\ldots)=x. \tag{\textbf{A3}, \textbf{A5}}\]
 The assertion for $n$ follows at once by induction.
 
 To prove item 8, suppose $x\leq y$. Axiom \textbf{A5} entails \[\fm(x)=\delta(x,\vec{0})\leq\delta(y,\vec{0})=\fm(y).\] Again, the assertion for $n$ is immediate by induction.
 
 For the final item, by item 8 and Lemma \ref{mv2} we have
  \begin{align*}
 \fm(x)\odot\fm(x)\leq \fm(1)\odot \fm(1).
 \end{align*}
 Using  item 6, we further have:
 \begin{align*}
 \fm(1)\odot \fm(1) &= \neg(\neg\fm(1)\oplus \fm(1))=\neg 1 = 0.
 \end{align*}
It follows that $\fm(x)\odot\fm(x)=0$, and the proof is complete.  
\end{proof}
\begin{lemma}\label{p1}
 The following equations hold in an arbitrary $\delta$-algebra, for each $n\in\N$.
\begin{enumerate}
 \item $\bigoplus_{i=1}^n \delta(0,\ldots,0,x_i,\vec{0})=\delta(x_1,\ldots,x_n,\vec{0})$.
 \item $\delta(\vec{x})=\delta(x_1,\ldots,x_n,\vec{0})\oplus\delta(0,\ldots,0,x_{n+1},x_{n+2},\ldots)$.
\end{enumerate} 
\end{lemma}
\begin{proof} 
In order to prove item 1, observe that the equation holds trivially for $n=1$, and it is true if $n=2$, by Lemma \ref{lem:1}.(4). Assuming that $n>2$ we show that, for all $k=1,\ldots,n-2$,
	\begin{align}\label{auxeq1}
	\bigoplus_{i=n-k}^n \delta(0,\ldots,0,x_i,\vec{0})=\delta(0,\ldots,0,x_{n-k},\ldots,x_n,\vec{0}).
	\end{align}
	If $k=1$, then 
	\begin{align*}
\bigoplus_{i=n-1}^n \delta(0,\ldots,0,x_i,\vec{0})&= \delta(0,\ldots,0,x_{n-1},\vec{0})\oplus \delta(0,\ldots,0,x_n,\vec{0}) \\
&= f_{\frac{1}{2^{n-2}}}(\delta(x_{n-1},\vec{0}))\oplus f_{\frac{1}{2^{n-2}}}(\delta(0,x_n,\vec{0})) \tag{Lemma \ref{lem:1}.(2)} \\
&=\delta(f_{\frac{1}{2^{n-2}}}(x_{n-1}),\vec{0})\oplus\delta(0,f_{\frac{1}{2^{n-2}}}(x_n),\vec{0}) \tag{Lemma \ref{lem:1}.(1)} \\
&=\delta(f_{\frac{1}{2^{n-2}}}(x_{n-1}),f_{\frac{1}{2^{n-2}}}(x_n),\vec{0}) \tag{Lemma \ref{lem:1}.(4)} \\
&=f_{\frac{1}{2^{n-2}}}(\delta(x_{n-1},x_n,\vec{0})) \tag{Lemma \ref{lem:1}.(1)} \\
&=\delta(0,\ldots,0,x_{n-1},x_n,\vec{0}). \tag{Lemma \ref{lem:1}.(2)} \end{align*}
Now suppose \eqref{auxeq1} is true for $1\leq k<n-1$. We prove that it is true for $k+1$. 
\begin{align*}
\bigoplus_{i=n-(k+1)}^n \delta(0,&\ldots,0,x_i,\vec{0})=\delta(0,\ldots,0,x_{n-k-1},\vec{0})\oplus \bigoplus_{i=n-k}^n \delta(0,\ldots,0,x_i,\vec{0}) \\
&=\delta(0,\ldots,0,x_{n-k-1},\vec{0})\oplus \delta(0,\ldots,0,x_{n-k},\ldots,x_n,\vec{0})\tag{inductive hypothesis} \\
&=f_{\frac{1}{2^{n-k-2}}}(\delta(x_{n-k-1},\vec{0}))\oplus f_{\frac{1}{2^{n-k-2}}}(\delta(0,x_{n-k},\ldots,x_n,\vec{0})) \tag{Lemma \ref{lem:1}.(2)} \\
&=\delta(f_{\frac{1}{2^{n-k-2}}}(x_{n-k-1}),\vec{0})\oplus \delta(0,f_{\frac{1}{2^{n-k-2}}}(x_{n-k}),\ldots,f_{\frac{1}{2^{n-k-2}}}(x_n),\vec{0}) \tag{Lemma \ref{lem:1}.(1)} \\
&=\delta(f_{\frac{1}{2^{n-k-2}}}(x_{n-k-1}),\ldots,f_{\frac{1}{2^{n-k-2}}}(x_n),\vec{0}) \tag{Lemma \ref{lem:1}.(4)} \\
&=f_{\frac{1}{2^{n-k-2}}}(\delta(x_{n-k-1},\ldots,x_n,\vec{0})) \tag{Lemma \ref{lem:1}.(1)} \\
&=\delta(0,\ldots,0,x_{n-k-1},\ldots,x_n,\vec{0}). \tag{Lemma \ref{lem:1}.(2)}
\end{align*} 
Applying \eqref{auxeq1} with $k=n-2$,
\begin{align*}
\bigoplus_{i=1}^n \delta(0,\ldots,0,x_i,\vec{0})&=\delta(x_1,\vec{0})\oplus \bigoplus_{i=2}^n \delta(0,\ldots,0,x_i,\vec{0}) \\
&=\delta(x_1,\vec{0})\oplus \delta(0,x_2,\ldots,x_n,\vec{0}) \\
&=\delta(x_1,\ldots,x_n,\vec{0}), \tag{Lemma \ref{lem:1}.(4)} 
\end{align*} so that item 1 is proved.

  For item 2, we proceed by induction on $n\in\N$. The identity holds for $n=1$ by Lemma $\ref{lem:1}.(4)$. Hence, let $n>1$. The proof is completed by the following computation.
  \begin{align*}
&\delta(x_1,\ldots,x_n,\vec{0})\oplus\delta(0,\ldots,0,x_{n+1},x_{n+2},\ldots)= \\
= & \ \bigoplus_{i=1}^{n-1} \delta(0,\ldots,0,x_i,\vec{0})\oplus \delta(0,\ldots,0,x_n,\vec{0})\oplus\delta(0,\ldots,0,x_{n+1},x_{n+2},\ldots) \tag{Lemma \ref{p1}.(1)}\\
=& \ \delta(x_1,\ldots,x_{n-1},\vec{0})\oplus\fmnn(\delta(x_n,\vec{0}))\oplus\fmnn(\delta(0,x_{n+1},x_{n+2},\ldots)) \tag{Lemmas \ref{p1}.(1), \ref{lem:1}.(2)} \\
=& \ \delta(x_1,\ldots,x_{n-1},\vec{0})\oplus \delta(\fmnn(x_n),\vec{0})\oplus\delta(0,\fmnn(x_{n+1}),\fmnn(x_{n+2}),\ldots) \tag{Lemma \ref{lem:1}.(1)} \\
=& \ \delta(x_1,\ldots,x_{n-1},\vec{0})\oplus \delta(\fmnn(x_n),\fmnn(x_{n+1}),\ldots) \tag{Lemma \ref{lem:1}.(4)} \\
=& \ \delta(x_1,\ldots,x_{n-1},\vec{0})\oplus \fmnn(\delta(x_n,x_{n+1},\ldots)) \tag{Lemma \ref{lem:1}.(1)} \\
=& \ \delta(x_1,\ldots,x_{n-1},\vec{0})\oplus \delta(0,\ldots,0,x_n,x_{n+1},\ldots) \tag{Lemma \ref{lem:1}.(2)} \\
=& \ \delta(\vec{x}). \tag{inductive hypothesis}
\end{align*}
\end{proof}
\begin{corollary}\label{cor:new}The following equations hold in an arbitrary $\delta$-algebra, for each $n\in\N$.
\begin{enumerate}
 \item $\delta(x_1,\ldots,x_n,\vec{0})=\bigoplus_{i=1}^{n}\fmi(x_i)$.
 \item $\delta(x,y,y,\ldots)=\fm(x)\oplus\fm(y)$.
 \item $\fm(x)\oplus\fm(x)=x$.
\end{enumerate} 
\end{corollary}
\begin{proof}
For item 1, note that \begin{align*}
   \delta(x_1,\ldots,x_n,\vec{0})&=\bigoplus_{i=1}^n \delta(0,\ldots,0,x_i,\vec{0}) \tag{Lemma \ref{p1}.(1)} \\
   &=\bigoplus_{i=1}^n  \fmi(x_i). \tag{Lemma \ref{lem:1}.(2)} \\
  \end{align*} 
 Item 2 is easily proved in the following way. \begin{align*}
        \delta(x,y,y,\ldots)&=\delta(x,\vec{0})\oplus\delta(0,y,y,\ldots) \tag{Lemma \ref{lem:1}.(4)} \\
        &=\fm(x) \oplus \fm(\delta(y,y,\ldots)) \tag{\textbf{A4}} \\
        &=\fm(x) \oplus \fm(y). \tag{\textbf{A3}}
        \end{align*}
 Finally, item 3 follows by an easy computation: \begin{align*}
\fm(x)\oplus\fm(x)&=\delta(x,x,\ldots) \tag{Corollary \ref{cor:new}.(2)} \\
&=x. \tag{\textbf{A3}}
\end{align*}
\end{proof}
\section{Every $\delta$-algebra is semisimple as an MV-algebra.}\label{s:semisimple}
Recall from Remark \ref{r:reduct} the forgetful functor $\forg{-} \colon \Va\to \MV$ that sends a $\delta$-algebra to its underlying MV-algebra. Our aim in this section is to show that for any $\delta$-algebra $A$, the MV-algebra $\forg{A}$ is semisimple (cf.\ Subsection \ref{ss:ideals}). This amounts to  the following: given a $\delta$-algebra $A$, prove that $\rad{\forg{A}}:=\bigcap{\Max{\forg{A}}}=\{0\}$. This will be the key fact needed to show, in the next section, that the functor  $\forg{-}$ is full. To improve legibility, in the rest of this section we blur the distinction between $\forg{A}$ and $A$.
\begin{lemma}\label{inrad}
 If $A$ is a $\delta$-algebra, and $x\in A$ satisfies $x\leq\fmn(1)$ for all $n\in\N$, then $x\in\rad{A}$.
\end{lemma}
\begin{proof}
Notice that \[x\leq\fm(1) \; \Leftrightarrow \; x\ominus \fm(1)=0 \; \Leftrightarrow \; x\odot \neg\fm(1)=0,\] so that, by Lemma \ref{lem:1}.(6), we have
 \begin{align}\label{firsteqaux} x\odot\fm(1)=0. \end{align} Let us fix an arbitrary $n\in\N$. Then $x\leq f_{\frac{1}{2^{n+1}}}(1)$ entails  \begin{align*}
nx&\leq \underbrace{f_{\frac{1}{2^{n+1}}}(1)\oplus\cdots \oplus f_{\frac{1}{2^{n+1}}}(1)}_{n \, \text{times}} \tag{Lemma \ref{mv2}}\\ 
&\leq \underbrace{f_{\frac{1}{2^{n+1}}}(1)\oplus\cdots \oplus f_{\frac{1}{2^{n+1}}}(1)}_{2^n \, \text{times}}.
\end{align*}
\begin{claim}\label{cl:inline}$\fm(1)=\underbrace{f_{\frac{1}{2^{n+1}}}(1)\oplus\cdots \oplus f_{\frac{1}{2^{n+1}}}(1)}_{2^n \, \text{times}}$.
\end{claim}
\begin{proof}[Proof of Claim \ref{cl:inline}]We prove  \[\underbrace{\fmn(x)\oplus\cdots\oplus\fmn(x)}_{2^m \, \text{times}}=f_{\frac{1}{2^{n-m}}}(x)\]
for each $x\in A$ and for each $m,n\in\N$ with $m<n$.  If $m=1$, then \begin{align*}
 \fmn(x)\oplus\fmn(x)&=\fm(\fmnn(x))\oplus\fm(\fmnn(x)) \\
 &=\fmnn(x). \tag{Corollary \ref{cor:new}.(3)}
\end{align*}
Now let $1<m<n$, and assume the statement holds for $m-1$. We see that \begin{align*}
\underbrace{\fmn(x)\oplus\cdots\oplus\fmn(x)}_{2^m \, \text{times}}&=(\underbrace{\fmn(x)\oplus\cdots\oplus\fmn(x)}_{2^{m-1} \, \text{times}})\oplus(\underbrace{\fmn(x)\oplus\cdots\oplus\fmn(x)}_{2^{m-1} \, \text{times}}) \\
&=f_{\frac{1}{2^{n-m+1}}}(x)\oplus f_{\frac{1}{2^{n-m+1}}}(x) \tag{inductive hypothesis} \\
&=\fm(f_{\frac{1}{2^{n-m}}}(x))\oplus \fm(f_{\frac{1}{2^{n-m}}}(x)) \\
&=f_{\frac{1}{2^{n-m}}}(x). \tag{Corollary \ref{cor:new}.(3)}
\end{align*}

\end{proof}
The claim entails $nx \leq \fm(1)$, so that 
by \eqref{firsteqaux} and Lemma \ref{mv2} we have \[x\odot nx\leq x\odot\fm(1)=0.\] 
Therefore \[x\odot nx=0  \, \Leftrightarrow \, nx\ominus\neg x=0 \, \Leftrightarrow \, nx\leq\neg x.\] Then Lemma \ref{mvrad} 
implies $x\in\rad{A}$.
\end{proof}
For the following lemma, recall Notation \ref{n:iterated sum}.
\begin{lemma}\label{radmez}
 If $A$ is a $\delta$-algebra and $x\in\rad{A}$ then, for all $n\in\N$, \[\fmn(2^n x)=x.\]
\end{lemma}
\begin{proof}It suffices to prove the assertion for $n=1$, as the general case follows by a trivial induction. We begin proving that the following holds for each $y \in A$.
\begin{align}\label{firststep}\fm(y\oplus y)=(y\oplus\fm(1))\ominus\fm(1).\end{align}
Indeed,
 \begin{align*}
\fm(y\oplus y)&= \fm(\neg(\neg y\ominus y)) \\
&=\fm(1\ominus ( (1\ominus y) \ominus y)) \\
&=\fm(1)\ominus ((\fm(1)\ominus \fm(y))\ominus \fm(y)) \tag{\textbf{A6}} \\
&=\fm(1)\ominus \neg((\neg\fm(1)\oplus\fm(y))\oplus\fm(y)) \\
&=\fm(1)\odot(\neg\fm(1)\oplus\fm(y)\oplus\fm(y)) \\
&=(\neg\fm(1)\oplus\fm(y)\oplus\fm(y))\ominus\neg\fm(1) \\
&=(\fm(1)\oplus\fm(y)\oplus\fm(y))\ominus\fm(1) \tag{Lemma \ref{lem:1}.(6)} \\
&=(y\oplus\fm(1))\ominus\fm(1).\tag{Corollary \ref{cor:new}.(3)}
\end{align*}

Also, applying Lemma \ref{mv4}.(2) with $x:=y\oplus\fm(1)$ and $y:=\fm(1)$, we see that  \begin{align}\label{step0}[(y\oplus\fm(1))\ominus\fm(1)]\oplus \fm(1)=y\oplus\fm(1)\end{align} holds in any $\delta
$-algebra.
It now suffices to show that the hypotheses of Lemma \ref{canc} are satisfied, that is \begin{align}\label{claim1}[(x\oplus\fm(1))\ominus\fm(1)]\odot \fm(1)=0,\end{align} and \begin{align}\label{claim2}x\odot\fm(1)=0,\end{align} 
for then an application of Lemma \ref{canc} to \eqref{step0} yields \[(x\oplus\fm(1))\ominus\fm(1)=x,\] so that by \eqref{firststep} we have $\fm(x\oplus x)=x$.     

Indeed, \eqref{claim1} follows by a straightforward MV-algebraic calculation. Since it is immediate that $\fm(1)\odot \fm(1)=0$, and $x\odot x=0$  by Lemma \ref{mv5},
a further application of Lemma \ref{mv5} proves \eqref{claim2}. 
\end{proof}
The next result is fundamental in proving that the radical ideal of a $\delta$-algebra is trivial.
\begin{lemma}\label{infinit}
 In any $\delta$-algebra $A$, if $\{x_i\}_{i\in\N}\subseteq \rad{A}$ then $\delta(\vec{x})\in\rad{A}$.
\end{lemma}
\begin{proof}
By Lemma \ref{inrad} it is sufficient to prove that, for all $n\in\N$, $\delta(\vec{x})\leq \fmn(1)$. Fix  $n\in \N$, and set $m:=n+1$. Then Corollary \ref{cor:new}.(3) entails at once
\begin{align}\label{eq1aux}\fmm(1)\oplus\fmm(1)=\fmn(1).\end{align} 
We claim that \begin{align}\label{blackclaim}\delta(x_1,\ldots,x_m,\vec{0})\leq \fmm(1).\end{align} 
Indeed, we have \begin{align*}
\delta(x_1,\ldots,x_m,\vec{0})&=\delta(\fmm(2^m x_1),\ldots,\fmm(2^m x_m),\vec{0}) \tag{Lemma \ref{radmez}} \\
&=\fmm(\delta(2^m x_1,\ldots,2^m x_m,\vec{0}) )\tag{Lemma \ref{lem:1}.(1)} \\
&\leq \fmm(1). \tag{Lemma \ref{lem:1}.(8)}
\end{align*}
 The proof is then completed by the following computation. \begin{align*}
\delta(\vec{x})&=\delta(x_1,\ldots,x_m,\vec{0})\oplus \delta(0,\ldots,0,x_{m+1},x_{m+2},\ldots) \tag{Lemma \ref{p1}.(2)} \\
&=\delta(x_1,\ldots,x_m,\vec{0})\oplus \fmm(\delta(x_{m+1},x_{m+2},\ldots)) \tag{Lemma \ref{lem:1}.(2)} \\
&\leq \fmm(1)\oplus \fmm(\delta(1,1,\ldots)) \tag{Eq.\ \ref{blackclaim}, \textbf{A5}, Lemma \ref{mv2}} \\
&=\fmm(1)\oplus \fmm(1) \tag{\textbf{A3}} \\
&=\fmn(1). \tag{Eq.\ \ref{eq1aux}}
\end{align*}
\end{proof}
We can finally prove the main result of this section.
\begin{theorem}\label{semisem}
 The underlying MV-algebra of any $\delta$-algebra is semisimple.
\end{theorem}
\begin{proof}
 Let $A$ be a $\delta$-algebra, and pick an arbitrary element $x\in\rad{A}$. Since $\rad{A}$ is an ideal, $nx\in\rad{A}$ for all $n\in\N$. Consider then the countable sequence \[\vec{v}:=\{2^i x\}_{i\in\N}\subseteq\rad{A}.\] Then $\delta(\vec{v})\in\rad{A}$ by Lemma \ref{infinit}. We next prove
\begin{align}\label{mainradeq}\delta(\vec{v})=x\oplus\delta(\vec{v}).\end{align}
Indeed,
\begin{align*} 
\delta(\vec{v})&=\delta(2x,\vec{0})\oplus\delta(0,2^2x,2^3x,\ldots) \tag{Lemma \ref{lem:1}.(4)}\\ 
&=\fm(2x)\oplus \fm(\delta(2^2x,2^3x,\ldots)) \tag{\textbf{A4}}\\
&=x\oplus \fm(\delta(2(2x),2(2^2x),\ldots)) \tag{Lemma \ref{radmez}}\\
&=x\oplus\delta(\fm(2(2x)),\fm(2(2^2x)),\ldots) \tag{\textbf{A2}} \\
&=x\oplus \delta(2x,2^2x,\ldots) \tag{Lemma \ref{radmez}}\\
&=x\oplus\delta(\vec{v}).
\end{align*}

Since $x,\delta(\vec{v})\in\rad{A}$, Lemma \ref{mv5} entails $x\odot \delta(\vec{v})=0$. Applying Lemma \ref{canc} to equation \eqref{mainradeq}, we conclude that $x=0$, i.e.\ $\rad{A}=\{0\}$.
\end{proof}
\section{The forgetful functor $\Va\longrightarrow\MV$ is full.}\label{s:etadelta}
In this section we continue to  blur the distinction between $\forg{A}$ and $A$, as in the previous one.
\begin{lemma}\label{fcomm} Let $A$ and $B$ be $\delta$-algebras, and let  $h\colon A\to B$ be any homomorphism of MV-algebras. Then $h$ preserves  the derived  operations $\fmn$, for all $n\in\N$.
\end{lemma}
\begin{proof}
In any MV-algebra $A$,  the following quasi-equation holds: if $x\odot x = 0$, $y\odot y = 0$, and $x\oplus x  = y\oplus y$, then $x=y$. To see this, represent $A$ as $\Gamma\Xi(A)$ by Lemma \ref{l:gamma}, and note that 
$x\odot x=(x+x-u)\vee 0$, so that $x\oplus x=x+x$, and $x=y$. We use this fact below.

It suffices to prove the assertion for $n=1$, for the rest follows by a straightforward induction. Pick $x\in A$. By Corollary \ref{cor:new}.(3) we have 
\[
h(\fm(x))\oplus h(\fm(x))=h(\fm(x)\oplus \fm(x))=h(x)
\]
and
 \[\fm(h(x))\oplus\fm(h(x))=h(x).\]
By Lemma \ref{lem:1}.(9), moreover,

\[
h(\fm(x))\odot h(\fm(x))=h(\fm(x)\odot\fm(x))=h(0)=0
\]
and
\[
\fm(h(x))\odot\fm(h(x))=0.
\] 
We conclude that $\fm(h(x))=h(\fm(x))$.
\end{proof}
\begin{lemma}\label{full01}
Let $A$ be a $\delta$-algebra. If $h\colon A\to\int$ is any homomorphism of MV-algebras, then $h$ is a $\delta$-homomorphism.
\end{lemma}
\begin{proof}
For all $\{x_i\}_{i\in\N}\subseteq A$, and for all $n\in\N$, we have the equality 
\begin{align*}
h(\delta(\vec{x}))&=h(\delta(x_1,\ldots,x_n,\vec{0}))\oplus h(\delta(0,\ldots,0,x_{n+1},x_{n+2},\ldots))   \tag{Lemma \ref{p1}.(2)}\\
&=h\left(\bigoplus_{i=1}^n \fmi(x_i)\right)\oplus h(\fmn(\delta(x_{n+1},x_{n+2},\ldots))) \tag{Corollary \ref{cor:new}.(1), Lemma \ref{lem:1}.(2)} \\
&=\bigoplus_{i=1}^n \frac{h(x_i)}{2^i} \oplus \frac{h(\delta(x_{n+1},x_{n+2},\ldots))}{2^n}. \tag{Lemma \ref{fcomm}}
\end{align*}
In $\int$,  $\bigoplus_{i=1}^n \frac{h(x_i)}{2^i}$ coincides with $\sum_{i=1}^n \frac{h(x_i)}{2^i}$, because \[\sum_{i=1}^n {\frac{h(x_i)}{2^i}}\leq\sum_{i=1}^n {\frac{1}{2^i}}=1-\frac{1}{2^n}<1.\] 
Moreover, since \[\left(\sum_{i=1}^n {\frac{h(x_i)}{2^i}}\right) + \frac{h(\delta(x_{n+1},x_{n+2},\ldots))}{2^n}\leq \left(1-\frac{1}{2^n}\right)+\frac{1}{2^n}=1\]
we have
 \begin{align}\label{limit2delta}\left(\sum_{i=1}^n {\frac{h(x_i)}{2^i}}\right) \oplus \frac{h(\delta(x_{n+1},x_{n+2},\ldots))}{2^n}=\left(\sum_{i=1}^n {\frac{h(x_i)}{2^i}}\right) + \frac{h(\delta(x_{n+1},x_{n+2},\ldots))}{2^n}.
 \end{align} 
Now, $h(\delta(\vec{x}))$ does not depend on $n\in\N$. Hence, using the equality proved at the beginning of this proof,
\begin{align*}
h(\delta(\vec{x}))&=\lim_{n\to\infty}{\left(\sum_{i=1}^n \frac{h(x_i)}{2^i} + \frac{h(\delta(x_{n+1},x_{n+2},\ldots))}{2^n}\right)} \tag{Eq.\ \ref{limit2delta}}\\
&=\lim_{n \to\infty}{\sum_{i=1}^n {\frac{h(x_i)}{2^i}}} + \lim_{n \to\infty}{\frac{h(\delta(x_{n+1},x_{n+2},\ldots))}{2^n}} \\
&=\sum_{i=1}^{\infty}{\frac{h(x_i)}{2^i}},
 \end{align*}
 which completes the proof.
\end{proof}
As a  consequence of Lemma \ref{full01}, we obtain: 
\begin{theorem}\label{thm:fullness}The functor $\forg{-}\colon\Va\to \MV$ is full.
\end{theorem}
\begin{proof}Arguing pointwise, Lemma \ref{full01} immediately yields the following: if $A$ is any $\delta$-algebra, and  $X$ is any compact Hausdorff space, then any MV-algebra homomorphism $A\to\C(X)$  is a $\delta$-homomorphism. Now consider an MV-algebra homomorphism $h\colon A\to B$ between $\delta$-algebras $A$ and $B$, and the component  $\eta_{B}\colon B \to \C(\Max{B})$ of the unit of the adjunction between $\MV$ and $\KH^{\rm op}$ of Section \ref{s:reflection}. The composition $\eta_{B}\circ h$ is then a $\delta$-homomorphism. Since $B$ is a semisimple MV-algebra by Theorem \ref{semisem}, $\eta_{B}$ is injective by Corollary \ref{cor:reflection}.(2), whence $h$ must preserve the operation $\delta$.
\end{proof}
We state an immediate and yet important consequence of the fact that $\forg{-}$ is fully faithful.
\begin{corollary}\label{cor:unique}Let $A$ be an MV-algebra. There is at most one structure of $\delta$-algebra on $A$ whose MV-algebraic reduct is $A$.\qed

\end{corollary}
\section{The Stone-Weierstrass Theorem.}\label{s:SW}
We saw in Corollary \ref{cor:reflection}.(2) that the MV-homomorphism $\eta_{A}\colon A\to\C(\Max{A})$ is injective if, and only if, the MV-algebra $A$ is semisimple. A characterisation of when  $\eta_{A}$  is surjective, given that it is injective, amounts to a Stone-Weierstrass Theorem for MV-algebras, which we prove in Lemma \ref{l:mvsw} assuming as Lemma \ref{l:sw} a version of the classical theorem. Using these ingredients we obtain, in Theorem \ref{t:SW-Delta}, a Stone-Weierstrass Theorem for $\delta$-algebras. 
\subsection{The Stone-Weierstrass Theorem for MV-algebras.}\label{ss:SW-MV}
For $X$ in $\KH$, a subset $A\subseteq \C(X,\R)$ is called \emph{separating}, or is said to 
\emph{separate the points of $X$}, if for each $x,y\in X$ such that $x\neq y$ there exists $f\in A$ such that  $f(x)\neq f(y)$.  
\begin{lemma}[The Stone-Weierstrass Theorem {\cite[Corollary 3]{Stone1948}, \cite[Theorem 7.29]{HR65}}]\label{l:sw} Let $X\neq\emptyset$ be a compact Hausdorff space, and let $G\subseteq \C(X,\R)$ be a unital $\ell$-subgroup, i.e.\ a sublattice subgroup containing the unit $1_{X}$. If $G$ is divisible \textup{(}as an Abelian group\textup{)} and separates the points of $X$, then $G$ is dense in $\C(X,\R)$ with respect to the uniform norm.
\qed\end{lemma}
In order to proceed, let us observe that  each unital $\ell$-group  $(H,u)$ is generated (as a group) by its unit interval $\Gamma(H,u)$. Indeed, given any non-negative element $h\in H$, since $u$ is a unit of $H$  there is $n\in \N$ such that $h\leq nu$. It is elementary that every $\ell$-group enjoys the Riesz decomposition property; thus (cf.\ \cite[Proposition 2.2]{Goodearl86}) there exist positive elements $h_1,\ldots, h_n$ of $H$ satisfying $h_i\leq u$ for each $i=1,\ldots, n$ and $h=\sum_{i=1}^{n}h_{i}$, as was to be shown.
We will freely use this fact in the sequel. Now, to translate Lemma \ref{l:sw} to MV-algebras we need the following  lemma about the functor $\Gamma$. 
\begin{lemma}\label{l:gammac(x)}Suppose the MV-algebra $A$ is a subalgebra of $\Gamma(G,u)$, for a unital $\ell$-group $(G,u)$.  Then the  $\ell$-group $G'$ generated by $A$ in $G$ is   isomorphic \textup{(}as a unital $\ell$-group\textup{)} to $\Xi(A)$, and $\Gamma(G',u)=A$.
\end{lemma}
\begin{proof}
Call $m$  the inclusion map $A \subseteq \Gamma(G,u)$. Further,  let $e\colon A\to \Gamma\Xi(A)$ be the map the sends $a\in A$ to the good sequence  $(a)$ of length $1$ in $\Xi(A)$. Now, using the adjunction $\Xi\dashv\Gamma$ 
granted by Lemma \ref{l:gamma}, $m$ induces a unique unital $\ell$-homomorphism $f\colon \Xi(A)\to G$ such that 
\begin{align}\label{eq:gamma}
\Gamma(f)\circ e =m.
\end{align}
 By Lemma \ref{l:gamma}, $e$ is an isomorphism, and therefore $\Gamma(f)$ is monic.
Hence, so is $f$, and the image $K:=f(\Xi(A))$ is unitally $\ell$-isomorphic to $\Xi(A)$. Since the unital $\ell$-group $K$ is generated by $\Gamma(K)$, it now suffices to show that $A=\Gamma(K)$.

To show $A\subseteq \Gamma(K)$, pick $a\in A$. Then $e(a)$ is such that $f(e(a))=a$ by \eqref{eq:gamma} together with the definition of  $\Gamma(f)$ as the restriction of $f$ to $\Gamma\Xi(A)$. This shows that $a \in K$. Since, moreover, $f$ preserves the unit, we have $a\in \Gamma(K)$.

To prove the converse inclusion, pick $x \in \Gamma(K)$. In other words, there exists a good sequence $(a_{1},\ldots,a_{l})$ such that (i) $f((a_{1},\ldots,a_{l}))=x$, and (ii) $0\leq f((a_{1},\ldots,a_{l}))\leq u$. Since $f$ is a unital 
$\ell$-group monomorphism, (ii) entails $(a_{1},\ldots,a_{l})\in \Gamma\Xi(A)$, and therefore $l=1$ by the definition of $\Xi(A)$. This proves $x=f(\,(a_{1})\,)=f(e(a_{1}))$. Again by  \eqref{eq:gamma}, we infer $x\in A$, and the proof is complete.
\end{proof}
\begin{remark}Lemma \ref{l:gammac(x)} is stated without proof in \cite[Lemma 2.1]{Cignolietal2004}, and it is part of the recent \cite[Theorem 1.2]{DubucPoveda15} --- where, as an additional bonus, the proof does not use the functor $\Xi$, i.e.\ good sequences.\end{remark}

We are now ready to prove:
\begin{lemma}[The MV-algebraic Stone-Weierstrass Theorem]\label{l:mvsw}Let $X\neq \emptyset$ be a compact Hausdorff space,  let $A$ be an MV-algebra, and suppose $A$ is a subalgebra of the MV-algebra $\C(X)$. If $A$ is closed under multiplication by  rational scalars in $\int$ and separates the points of $X$, then $A$ is dense in $\C(X)$ with respect to the uniform norm.
\end{lemma}
\begin{proof}By Lemma \ref{l:gammac(x)}, the $\ell$-subgroup $G$ of $\C(X,\R)$ generated by $A$ is isomorphic to $\Xi(A)$ (so that $G$ contains $1_{X}$). 

We \emph{claim} that $G$ is divisible, i.e.\ for each $f\in G$, and each $n\in\N$, there is $g\in G$ such that $ng=f$.
Each element of $G$ may be expressed as the difference of two non-negative elements, thus it suffices to settle the \emph{claim} for $f\in G$ with $f\geq 0$.
Since $\Gamma(G,1_{X})=A$ by Lemma \ref{l:gammac(x)}, there are  finitely many elements $f_{1},\ldots,f_{l}\in A$ such that 
$f=\sum_{i=1}^{l}f_{i}$.
Given $n\in\N$, set $g:=\sum_{i=1}^{l}\frac{f_{i}}{n}$. Since $A$ is closed under multiplication by $\frac{1}{n}$, then $\frac{f_{i}}{n}\in A$, and  $g\in G$. Hence, $ng=f$ and the \emph{claim} is proved.

It is obvious that $G$ separates the points of $X$, because $A$ does. Then, by Lemma \ref{l:sw}, $G$ is dense in $\C(X,\R)$ with respect to the uniform norm $\|\cdot\|$. 
Since, for each $g\in G$ and each $f\in \C(X)$, it is elementary that $\|f-g\|\geq \|f-((g \wedge 1_{X})\vee 0_X)\|$, $A$ is dense in $\C(X)$.  
\end{proof}
\subsection{The Stone-Weierstrass Theorem for $\delta$-algebras.}\label{ss:SW-Delta} %
We begin with an adaptation to MV-algebras  of a standard fact on uniform approximations.
\begin{lemma}\label{increasing}
Let $X\neq \emptyset$ be a compact Hausdorff space,  and let $A$ be  a dense MV-subalgebra of the MV-algebra $\C(X)$. For each $f\in\C(X)$, there exists an increasing 
sequence $\{f_i\}_{i\in\N}\subseteq A$ which converges uniformly to $f$.
\end{lemma}
\begin{proof}Let $G$ be the $\ell$-subgroup of $\C(X,\R)$ generated by $A$. We first prove that $G$ is dense in $\C(X,\R)$. Each element of $\C(X,\R)$ may be expressed as the difference of two non-negative elements, therefore we can safely assume that $f\in \C(X,\R)$ satisfies $f\geq 0$.
Since $\Gamma(\C(X,\R),1_{X})=\C(X)$, there exist finitely many elements $f_{1},\ldots,f_{l}\in \C(X)$ such that $f=\sum_{i=1}^{l}f_{i}$.
By the density of $A$ in $\C(X)$, for each $i$ there exists a sequence $\{a^{i}_{j}\}_{j\in\N}\subseteq A$ that converges to $f_{i}$ uniformly. For each $j\in \N$, set
$a_{j}:=\sum_{i=1}^l a_{j}^{i}$,
so that $\{a_{j}\}_{j\in \N}\subseteq G$. An application of the triangular inequality now shows that $\{a_{j}\}_{j\in \N}$ converges to $f$ uniformly, so that $G$ is dense in $\C(X,\R)$.

 Now, let us pick an arbitrary $f\in\C(X)$. Consider the strictly decreasing sequence $\{r_i=\frac{1}{2^i}\}_{i\in\N}$, and define two more sequences $\{s_i\}_{i\in\N}$, $\{t_i\}_{i\in\N}\subseteq \R$ by setting  \[s_i:= \frac{1}{2}(r_i+r_{i+1})=\frac{3}{2^{i+2}}, \; \; t_i:= \frac{1}{2}(r_i-r_{i+1})=\frac{1}{2^{i+2}}.\] Then $s_i+t_i=\frac{1}{2^i}=r_i$ and $s_i-t_i=\frac{1}{2^{i+1}}=r_{i+1}$. 
 By the density of $G$, for each $i \in \N$ there is
 $g_i\in G$ such that $\|g_i-(f- s_i1_{X})\|<t_i$. This is equivalent to saying that, for all $x\in X$, 
 \[f(x)-\frac{1}{2^{i-1}}<g_i(x)<f(x)-\frac{1}{2^i}.\]
  Note that 
 \[f(x)-\frac{1}{2^{i-1}}<g_i(x)<f(x)-\frac{1}{2^i}<g_{i+1}(x)<f(x)-\frac{1}{2^{i+1}},\]
  i.e.\ $\{g_i\}_{i\in\N}\subseteq G$ is a strictly  increasing sequence. Set $f_i:= g_i\vee 0$. Then $f_i\in A$ since, by Lemma \ref{l:gammac(x)}, $\Gamma(G,1_{X})=A$. It is elementary that $\{f_i\}_{i\in\N}\subseteq A$ is an  increasing sequence. Moreover, the latter converges to $f$ uniformly because, for all $i\in\N$, \[\|f_i-f\|\leq \|g_i-f\|=\|(g_i-f+s_i1_{X})-s_i1_{X}\|\leq \|g_i-(f-s_i1_{X})\|+\|-s_i1_{X}\|<\frac{1}{2^{i+2}}+\frac{3}{2^{i+2}}=\frac{1}{2^i}.\]
\end{proof}
The following argument is essentially due to Isbell \cite{Isbell}. 
\begin{lemma}\label{isbgen}
Let $A$ be a $\delta$-subalgebra of $\C(X)$ for some $X\neq\emptyset$ in $\KH$, and let $f\in \C(X)$. Suppose that $\{s_i\}_{i\in\N}\subseteq A$ satisfies the following properties.
\begin{enumerate}
\item $\{s_i\}_{i\in\N}$ is an increasing sequence. 
\item $\{s_i\}_{i\in\N}$  converges uniformly to $f$.
\item $\|s_1\|\leq\frac{1}{2}$, and for all $i\geq 2$,   $\|s_i-s_{i-1}\|\leq \frac{1}{2^i}$.
\end{enumerate}
Then 
\[
f=\delta(2s_1,2^2(s_2\ominus s_1),\ldots, 2^i(s_i\ominus s_{i-1}),\ldots).
\]
In particular, $f\in A$.
\end{lemma}
\begin{proof}Set $s_0:= 0\in A$, and 
\[\vec{s}:=\{2^i(s_i\ominus s_{i-1})\}_{i\in\N}=\{2s_1,2^2(s_2\ominus s_1),\ldots, 2^i(s_i\ominus s_{i-1}),\ldots\}.\] 
The assumptions in items 1 and 3 entail through an elementary computation that
\[
2^i(s_i\ominus s_{i-1})=\sum_{j=1}^{2^{i}}(s_i-s_{i-1}).
\]
Then,
 \begin{align*}
\delta(\vec{s})&=\sum_{i=1}^{\infty}{\frac{2^i(s_i- s_{i-1})}{2^i}}\\
&=\lim_{n\to\infty}{\sum_{i=1}^n {(s_i-s_{i-1})}} \\
&=\lim_{n\to\infty}{s_n} = f,\end{align*}
 where the last equality holds by  the assumption in item 2.
 \end{proof}
We are now ready to prove:
\begin{theorem}[The Stone-Weierstrass Theorem for $\delta$-algebras]\label{t:SW-Delta}Let $X$ be a compact Hausdorff space, and let $A\subseteq\C(X)$ be a $\delta$-subalgebra that separates the points of $X$. Then $A=\C(X)$.
\end{theorem}
\begin{proof}If $X=\emptyset$ then $A$ is trivial and the theorem obviously holds. Assume $X\neq \emptyset$. Note that $A$ is closed under multiplication by scalars in $\int$.
Indeed, let  $r\in
\int$, and let $f\in A$. Pick a binary expansion $\{r_i\}_{i\in\N}\in\{0,1\}^{\N}$ of $r$, and define $\vec{f}:=\{f_i\}_{i\in\N}\subseteq A$ by setting $f_i:=0$ if $r_i=0$, and $f_i:=f$ if $r_i=1$. Consequently, \[\delta(\vec{f})=\sum_{i=1}^{\infty}{\frac{f_i}{2^i}}=\sum_{i=1}^{\infty}
{\frac{r_i f}{2^i}}=f\cdot\sum_{i=1}^{\infty}{\frac{r_i}{2^i}}=rf \in A,\]   
as was to be shown.
It now follows from Lemma  \ref{l:mvsw} that $A$ is dense in $\C(X)$.  We prove  $A=\C(X)$.

For any $f\in\C(X)$, by Lemma \ref{increasing} there is an increasing sequence $\{f_j\}_{j\in\N}\subseteq A$ that converges uniformly to $f$. Consider the sequence $\{\frac{f_j}{2}\}_{j\in\N}\subseteq A$, and 
note that $\|\frac{f_j}{2}\|\leq\frac{1}{2}$ for all $j\in \N$. Extract a subsequence $\{s_i\}_{i\in\N}\subseteq\{\frac{f_j}{2}\}_{j\in\N}$ satisfying $\|s_i-s_{i-1}\|\leq \frac{1}{2^i}$ for all $i\in\N$. Now $\frac{f}{2}\in A$ 
because $\{s_i\}_{i\in\N}$ satisfies the hypotheses of Lemma \ref{isbgen}, and thus $f \in A$. 
 \end{proof}

\section{Main result.}\label{s:main}
We now return to the  forgetful functor
$\forg{-} \colon \Va\to \MV$ that sends a $\delta$-algebra to its underlying MV-algebra (Remark \ref{r:reduct}). Recall that $\forg{-}$ is full (Theorem \ref{thm:fullness}); it is obviously faithful. Let us write $\forg{\Va}$ for the full subcategory of $\MV$ on the objects in the image of $\forg{-}$. We  can now compare $\MV_{\C}$ (Notation \ref{not:mvc}) and $\forg{\Va}$. 
\begin{theorem}\label{t:deltaismvc}
$\forg{\Va}=\MV_{\C}$.
\end{theorem}
\begin{proof}It suffices to prove that $\forg{A}$ is in $\MV_{\C}$ for each $\delta$-algebra $A$. Consider then the component $\eta_{\forg{A}}$ as in \eqref{eq:eta}, with the aim of showing that it is an isomorphism; this amounts to showing that $\eta_{\forg{A}}$ is bijective. By Theorem \ref{semisem} together with Corollary \ref{cor:reflection}.(2), $\eta_{\forg{A}}$ is injective. Further, note that $\eta_{\forg{A}}$ is a $\delta$-homomorphism by Theorem \ref{thm:fullness}, and that its  image  is a separating subset of $\C(\Max{\forg{A}})$:  given $\m\neq \mathfrak{n} \in \Max{\forg{A}}$, there must be $a\in\m\setminus\mathfrak{n}$, and  $\widehat{a}\colon \Max{\forg{A}}\to \int$ separates $\m$ from $\mathfrak{n}$. Therefore, Theorem \ref{t:SW-Delta} applies to yield the surjectivity of $\eta_{\forg{A}}$.
\end{proof}

Recall from Section \ref{s:reflection} the adjunction $\Max\dashv\C\colon\KH^{\rm op}\to\MV$. By Theorem \ref{t:deltaismvc} we obtain a  forgetful functor 
$\forg{-}\colon\Va\to\MV_{\rm C}$. This functor has an inverse $F\colon \MV_{\C}\to \Va$. Indeed, given an MV-algebra $A$ in $\MV_{\C}$, the arrow $\eta_{A}$ is an 
MV-isomorphism and  $A$ inherits from $\C(\Max{A})$ a structure of $\delta$-algebra whose MV-algebraic reduct is $A$. Such a structure is unique by Corollary \ref{cor:unique}.  Call  $F(A)$ the resulting
$\delta$-algebra. By Theorem \ref{thm:fullness}, $F$ extends to a functor $F\colon \MV_{\rm C}\to\Va$ in the obvious manner. Thus, $\Va$ 
and $\MV_{\rm C}$ are isomorphic categories. Now, recall from Corollary \ref{cor:reflection}.(1) that the adjunction $\Max\dashv\C$ restricts to an equivalence between $\MV_{\rm C}$ and $\KH^{\rm op}$. Continuing to denote by $\Max$ and $\C$ 
the restricted functors, this proves: {\it The functors $F\circ \C\colon \KH^{\rm op}\to\Va$ and $\Max\circ \forg{-}\colon \Va\to \KH^{\rm op}$ form an equivalence of categories}.

Note that the composition $F\circ \C \colon \KH^{\rm op}\to\Va$ coincides, up to variance, with the contravariant 
$\hom$-functor $\hom_{\KH}{(-,\int)} \colon \KH\to \Va$. 
Moreover, the composition  $\Max\circ \forg{-}\colon \Va\to \KH^{\rm op}$ coincides, up to variance, with the contravariant 
$\hom$-functor $\hom_{\Va}{(-,\int)}\colon \Va \to \KH$, because of Lemma \ref{full01}. Our main result can finally be stated as follows.
\begin{theorem}\label{thm:main}The contravariant functors  $\hom_{\KH}{(-,\int)} \colon \KH\to  \Va$ and $\hom_{\Va}{(-,\int)}\colon \Va \to \KH$ form a dual equivalence of categories. \qed
\end{theorem}

\medskip
\begin{acks}
We gratefully acknowledge  partial support by the Italian FIRB ``Futuro in Ricerca'' grant no.\ RBFR10DGUA, and by the Italian PRIN grant ``Metodi Logici per 
il Trattamento dell'Informa\-zio\-ne". The first-named author is much indebted to Dirk Hofmann and to Ioana Leu\c{s}tean for several  conversations on 
infinitary varieties and on MV-algebras dually equivalent to compact Hausdorff spaces, respectively, from which he has learned a lot. He is also grateful to 
Alexander Kurz for directing him towards the references \cite{PelletierRosicky89, PelletierRosicky93}. We both  would like to thank Ji\v{r}\'i Rosick\'y, who in 
e-mail communications helped us clarifying the current status of the axiomatisation problem for the category opposite to compact Hausdorff spaces, and kindly provided us with historical information concerning the solution of Bankston's conjecture (cf.\ Footnote \ref{fn:bank}). We are also greatly indebted to Roberto Cignoli and Eduardo Dubuc for several useful comments on a preliminary version of this paper that led us to improve the presentation of our results. 
Finally, we are most grateful to two anonymous referees, whose comments further helped us to achieve a more effective presentation of our results.
Needless to say, all remaining  deficiencies in presentation are entirely due to us.
\end{acks}

\bibliographystyle{plain}

\end{document}